\documentclass[a4paper]{amsart}

\usepackage{amsthm}
\usepackage{epsfig}
\usepackage{alltt}
\usepackage[active]{srcltx}
\usepackage{newlfont}
\usepackage{amsmath}
\usepackage{amssymb}
\usepackage{amsfonts}
\usepackage{amscd}
\usepackage{verbatim}
\usepackage{changebar}
\usepackage[OT2,T1]{fontenc}
\usepackage{enumerate}
\usepackage{prettyref}
\usepackage{hyperref}
\usepackage{mathrsfs}
\usepackage{color}
\usepackage{ulem}
\input{xy}
\xyoption{all}

\newcommand{\Ao}{\ensuremath{\mathbb{A}^1}}
\DeclareSymbolFont{cyrletters}{OT2}{wncyr}{m}{n}
\DeclareMathSymbol{\Be}{\mathalpha}{cyrletters}{"42}
\newcommand{\Bmot}{\ensuremath{\operatorname{DM}_{\Be}}}
\newcommand{\Bmotc}{\ensuremath{\operatorname{DM}_{\Be,c}}}

\newcommand{\et}{\text{\'et}}
\newcommand{\DMT}{\ensuremath{\operatorname{MTDer}}}
\newcommand{\PMT}{\ensuremath{\operatorname{MTPer}}}

\newrefformat{thm}{\hyperref[{#1}]{Theorem~\ref*{#1}}}
\newrefformat{lem}{\hyperref[{#1}]{Lemma~\ref*{#1}}}
\newrefformat{prop}{\hyperref[{#1}]{Proposition~\ref*{#1}}}
\newrefformat{defin}{\hyperref[{#1}]{Definition~\ref*{#1}}}
\newrefformat{cor}{\hyperref[{#1}]{Corollary~\ref*{#1}}}
\newrefformat{rem}{\hyperref[{#1}]{Remark~\ref*{#1}}}
\newrefformat{ex}{\hyperref[{#1}]{Example~\ref*{#1}}}

\newtheorem{mainthm}{Theorem}
\newtheorem{theorem}{Theorem}[section]

\newtheorem{lemma}[theorem]{Lemma}
\newtheorem{proposition}[theorem]{Proposition}
\newtheorem{corollary}[theorem]{Corollary}

{
\theoremstyle{definition}
\newtheorem{definition}[theorem]{Definition}
\newtheorem{convention}[theorem]{Notational convention}
\theoremstyle{remark}
\newtheorem{remark}[theorem]{Remark}
\newtheorem{example}[theorem]{Example}
\newtheorem{Bemerkungl}[theorem]{}}

\newcommand{\op}{\operatorname}

\newcommand{\ra}{\rightarrow}

\newcommand{\sra}{\twoheadrightarrow}
\newcommand{\hra}{\hookrightarrow}

\newcommand{\sira}{\stackrel{\sim}{\rightarrow}}
\newcommand{\sirra}{\stackrel{\approx}{\rightarrow}}

\newcommand{\da}{\downarrow}

\newcommand{\RA}{\Rightarrow}

\newcommand{\Bl}[1]{{\mathbb{#1}}}

\newcommand{\DC}{\Bl{C}}

\newcommand{\DZ}{\Bl{Z}}

\newcommand{\DQ}{\Bl{Q}}

\newcommand{\pdef}{\mathrel{\mathop:}=}

\begin{document}
\title[Perverse motives and $\mathcal O$]{Perverse motives and graded 
derived category \texorpdfstring{$\mathcal{O}$}{O}}
\date{October 2015}
\author{Wolfgang Soergel and Matthias Wendt}
\thanks{Wolfgang Soergel was supported by the DFG priority program SPP
  1388.  Matthias Wendt was partially supported  by the DFG SFB TR 45.}
\address{Wolfgang Soergel, Mathematisches Institut,
Albert-Ludwigs-Uni\-ver\-si\-t\"at Freiburg, Eckerstra\ss{}e 1, 79104, 
  Freiburg im Breisgau, Germany}
\email{wolfgang.soergel@math.uni-freiburg.de}
\address{Matthias Wendt, Fakult\"at f\"ur Mathematik, Universit\"at
  Duisburg-Essen, Thea-Leymann-Strasse 9, 45127 Essen}
\email{matthias.wendt@uni-due.de}

\subjclass[2010]{14C15, 14M15, 17B10, 22E47}

\keywords{mixed Tate motives, perverse sheaves, mixed geometry,
  category $\mathcal{O}$} 

\begin{abstract}
For a variety with a Whitney stratification by affine spaces, we study
categories of motivic sheaves which are constant mixed Tate along the
strata. We are particularly interested in those cases where the category of mixed Tate motives over a point is equivalent to the category of finite-dimensional bigraded vector spaces. Examples of such situations include rational motives on varieties over finite fields and modules over the spectrum representing the semisimplification of de Rham cohomology for varieties over the complex numbers. We show that our categories of stratified mixed Tate motives have a natural weight structure. Under an additional assumption of pointwise purity for objects of the heart, tilting gives an equivalence between  stratified mixed Tate sheaves and the bounded homotopy category of the heart of the weight structure. Specializing to the case of flag varieties, we find natural geometric interpretations of graded category $\mathcal O$ and  Koszul duality. 
\end{abstract}

\maketitle 
\setcounter{tocdepth}{1}
\tableofcontents

\section{Introduction}

In \cite{BGi} Beilinson and Ginzburg laid out a vision how ``mixed geometry'' should allow to construct graded versions of the BGG-category $\mathcal{O}$, and why these graded versions should be governed by Koszul rings. Motivated by inversion formulas for Kazhdan--Lusztig polynomials, they also conjectured these Koszul rings to be their own Koszul duals. This was pushed through in  \cite{BGSo}; however, the beauty of the original ideas got kind of obscured by difficulties stemming from the fact, that in all realizations of mixed geometry available back then, there would be some (unwanted) non-trivial extensions between Tate motives. In the current paper, we want to
show how recent advances in constructing triangulated categories of motives and
motivic six functor formalisms allow to clear away this difficulty and
realize the original vision in its full beauty.  

We will define, for a given motivic triangulated category $\mathscr{T}$ and a  ``Whitney--Tate'' stratified variety $(X,\mathcal{S})$ over an arbitrary field, a triangulated $\mathbb{Q}$-linear category $\DMT_{\mathcal{S}}(X,\mathscr{T})$ of $\mathscr{T}$-motives which are constant mixed Tate along the strata, called {\bf stratified mixed Tate motives}. Of particular interest for 
our applications are those motivic triangulated categories $\mathscr{T}$ where the category of mixed Tate $\mathscr{T}$-motives over the base field is equivalent to the category of bigraded $\mathbb{Q}$-vector spaces of finite dimension, viewed as the derived category of the abelian category of finite-dimensional $\mathbb{Z}$-graded $\mathbb{Q}$-vector spaces. This happens for motives with rational coefficients over finite fields and for the category of motives associated to the enriched Weil cohomology theory given by semisimplified Hodge realization \cite{drew}. These two cases are related to the $\ell$-adic and Hodge module approximations to mixed geometry previously available.

For representation-theoretic purposes, the case of $X=G/P$ with the stratification $\mathcal{S}=(B)$ by Borel orbits is particularly interesting. We show in this case, that the category $\DMT_{(B)}(G/P)$ of stratified mixed Tate motives on the flag variety carries two interesting additional structures: a weight structure whose heart is related to categories of Soergel modules, and a perverse t-structure whose heart is a graded version of category $\mathcal{O}$. In particular, we construct an equivalence of triangulated $\DQ$-categories 
$$\DMT_{(B)}(G/B)\cong  \op{Hot}^{\op{b}}(C\op{-SMod}^\DZ_{\op{ev}}),$$
where $C=\op{H}^*(G/B,\mathbb{Q})$ denotes the cohomology ring with
rational coefficients of the complex flag manifold sometimes called
the coinvariant algebra, $C\op{-SMod}^\DZ\subset C\op{-Mod}^\DZ$
denotes a full subcategory of the category of all graded finite-dimensional $C$-modules sometimes called the category of Soergel modules, and $C\op{-SMod}^\DZ_{\op{ev}}$ denotes the full subcategory of Soergel modules concentrated in even degrees only. Put another way, the category $\DMT_{(B)}(G/B)$ is, up to adding a root of the Tate twist, equivalent to the bounded derived category of the graded version of the principal block $\mathcal{O}_0$ of category $\mathcal{O}$ constructed in \cite{BGSo}. The idea of such a geometrical or even motivic construction was already clearly present in the seminal preprint \cite{BGi} of Beilinson and Ginzburg. 

Let us discuss the relation of our work to what has been done already.  
The first geometric realizations of the (not yet derived) graded category $\mathcal O$ were constructed in \cite{BGSo}. A geometric realization of the  graded derived category $\mathcal{O}$ has been constructed by Achar and Riche \cite{AchR}. Another approach by ``winnowing'' categories of mixed Hodge modules in the sense of Saito is worked out by Achar and Kitchen in \cite{AK}. Modular coefficient realizations are discussed in \cite{ARM1,ARM2}. 

These approaches, using Hodge modules, \'etale or $\ell$-adic sheaves, are technically demanding due to problems with non-semisimplicity of the corresponding categories of sheaves for the one-point flag variety. This is the main motivation for us to suggest yet another realization of the graded derived category $\mathcal{O}$. While  the theory of motives is also built on technically demanding foundations, our point in the present paper is that at least such problems as the non-semisimplicity of Frobenius actions disappear, and the geometric construction of the graded derived category $\mathcal{O}$ is clarified and simplified considerably by using true motives. We hope that our explanations contribute to a better understanding of the original vision laid out in the work \cite{BGi} of Beilinson and Ginzburg. We also expect that the use of mixed motivic categories will turn out to be fruitful in a lot of other instances where geometric representation theory relies on ``mixed geometry''. For example, in a joint work in progress with Rahbar Virk, we will discuss how Borel-equivariant motives can be used to establish motivic versions of the results from \cite{virk} and construct a very natural geometric categorification of the Hecke algebra. 

\subsection{Motivic triangulated categories and stratified mixed Tate
  motives}

Let us now outline in more detail the constructions and results to be presented in this work. The most important technical tools used in the paper are the recent works on categories of motives over an arbitrary base and their six-functor formalism: \cite{ayoub:thesis1,ayoub:thesis2}, \cite{cisinski:deglise} and \cite{drew}. With these motivic categories and their six-functor formalism available, many of the standard arguments that have been developed in geometric representation theory can be adapted to the setting of motives, be it \'etale motives, Beilinson  motives or motives with coefficients in enriched mixed Weil cohomology theories. For the applications we have in mind we restrict to categories of mixed Tate motives, which are much better understood, due to the work of Levine \cite{levine:nato,levine:tata}, Wildeshaus \cite{wildeshaus:cras1} and others. We include two sections discussing these technical foundations: \prettyref{sec:cdbmot} is a very abridged recollection of basics on triangulated categories of motives, and \prettyref{sec:mtm} recalls relevant facts about mixed Tate motives.  

With these tools in hands, we construct in \prettyref{sec:dmt}, for a motivic triangulated category $\mathscr{T}$ and a stratified variety $(X,\mathcal{S})$, an analogue of the category of sheaves which are constant along strata. This category, denoted by  $\DMT_{\mathcal{S}}(X,\mathscr{T})$, is called the \emph{category of stratified mixed Tate $\mathscr{T}$-motives}, and consists of those motives which are constant mixed Tate along the strata. For this category to be well-behaved, one needs as in \cite{BGSo} or \cite{wildeshaus:intermediate} a condition ``Whitney--Tate'', which ensures that the extension and restriction functors preserve mixed Tate motives. This condition is satisfied in a large number of cases, including, in particular, partial flag varieties stratified by Borel orbits.

While the construction of $\DMT_{\mathcal{S}}(X)$ works in great generality, we usually work in a more restrictive setting where the underlying motivic triangulated category $\mathscr T$ satisfies additional conditions, cf. Convention~\ref{conditions}: 
\begin{enumerate}
\item one condition is called \emph{weight condition} and requires the existence of suitably compatible weight structures on the motivic categories, 
\item the other  condition is called \emph{grading condition} and requires the category $\DMT(k,\mathscr{T})$ of mixed Tate $\mathscr{T}$-motives over the  base field $k$ to be equivalent to the derived category of the category of $\mathbb{Z}$-graded vector spaces. 
\end{enumerate} 
The grading condition implies that the category $\DMT_{\mathcal{S}}(X)$ can be described very explicitly in terms of the combinatorics of the stratification $\mathcal{S}$. These conditions are satisfied in two important cases: rational \'etale or Beilinson motives over a finite field $\mathbb{F}_q$, 
and motives with coefficients in the semisimplified Hodge realization over $\mathbb{C}$. 

\subsection{Weight structures}

The first block of results in our paper concerns a weight structure on the category of stratified mixed Tate motives. Weight arguments have been used a lot in geometric representation theory, in particular in the framework of mixed geometry. The very recently introduced weight structures alias co-$t$-structures, due independently to Bondarko \cite{bondarko:ktheory} and Pauksztello, are a convenient framework for formalising such weight arguments. The following result establishes the existence of a weight structure on stratified mixed Tate motives,  cf. \prettyref{prop:hebbswz}, and \prettyref{thm:Twer}. It follows rather easily from the existence of weight structures on Beilinson motives, as constructed by H{\'e}bert \cite{hebert} and Bondarko \cite{bondarko:imrn}. 

\begin{mainthm}
\label{thm:main1X} 
Let $k$ be a field, and let $\mathscr{T}$ be a motivic triangulated category over $k$ satisfying the weight condition, cf. Convention~\ref{conditions}. Let $(X,\mathcal{S})$ be an affinely Whitney--Tate stratified $k$-variety in the sense of \prettyref{defin:cellvar} and \prettyref{defin:WT}.
\begin{enumerate}
\item  The category $\DMT_{\mathcal{S}}(X)$ carries a weight structure 
  in the sense of Bondarko.  This weight
  structure is uniquely determined by the requirement that for the 
  inclusion $j_s:X_s\to X$ of a stratum, the functors $j_s^\ast$ and
  $j_s^!$ preserve   non-positivity and non-negativity of weights,
  respectively. 
\item 
Assume that $\mathscr{T}$ also satisfies the 
grading condition of Convention~\ref{conditions}, that all objects of the heart $\DMT_{\mathcal{S}}(X)_{w=0}$ are  pointwise pure in the sense of Definition \ref{poip}, and that  $\DMT_{\mathcal{S}}(X)$ can be embedded as full subcategory of a localization of the derived category of some abelian category. Then the tilting functor of \prettyref{prop:rvt} induces an equivalence 
$$
\op{Hot}^{\op{b}}(\DMT_{\mathcal{S}}(X)_{w=0})
\stackrel{\approx}{\longrightarrow} 
\DMT_{\mathcal{S}}(X)
$$
between the category of stratified mixed Tate motives on $X$ and the bounded homotopy category of the heart of its weight structure.
\end{enumerate}
\end{mainthm}

\begin{remark}
The slightly awkward condition in (2) is one possibility to ensure the applicability of the tilting result in Proposition~\ref{prop:rvt}. This condition is satisfied in the  cases which for us are the most interesting: rational motives (\'etale or Beilinson) over a finite field, and motives with coefficients in the semisimplification of the Hodge realization over $\mathbb{C}$. The explicit tilting result \prettyref{prop:rvt} is the only place in the whole paper where we need information beyond the axiomatics of motivic triangulated categories --- for the tilting result to hold we need some more information on how the motivic triangulated category at hand is constructed. All other results in the paper only use the axiomatics of motivic triangulated categories. 
\end{remark}

We also adapt pointwise purity arguments of Springer \cite{Spp} and a full faithfulness result of Ginzburg \cite{Gi} to describe the heart of the above weight structure. The following result explicitly describes the category of stratified mixed Tate motives in a case of representation-theoretic interest in terms of Soergel modules, cf. \prettyref{cor:bsgen}, \prettyref{lem:bswz},  \prettyref{thm:ffbs} and \prettyref{cor:ttb}.

\begin{mainthm} 
\label{thm:main1O}
Assume that $k$ is a field and $\mathscr{T}$ is a motivic triangulated category over $k$ satisfying the weight condition. Let $G\supset P\supset B\supset T$ be a split reductive group over  $k$ with a parabolic, a Borel and a split maximal torus all defined over $k$. Let $X=G/P$ be the corresponding partial flag variety with $\mathcal{S}=(B)$ its stratification by $B$-orbits. Then we have:
\begin{enumerate}
\item  The heart $\DMT_{(B)}(G/P)_{w=0}$ of the weight structure from \prettyref{thm:main1X}  is generated, as idempotent complete additive subcategory of $\mathscr{T}(G/P)$, by motives of Bott--Samelson resolutions of Schubert varieties in $G/P$. These are pointwise pure. 
\item Let $k=\mathbb{F}_q$ be a finite field and let $\mathscr{T}$ be the motivic triangulated category of (\'etale or Beilinson) motives with rational coefficients. Then a suitable hypercohomology functor induces an equivalence  of categories 
$$
\DMT_{(B)}(G/P)_{w=0}\sirra
\op{H}^*(G/P)\op{-SMod}^{\mathbb{Z}}_{\op{ev}} 
$$
between the heart of the weight structure and the category of even Soergel
modules over the cohomology ring of $G/P$. This equivalence extends to an equivalence 
$$
\DMT_{(B)}(G/P)
\sirra \op{Hot}^{\op{b}}(\op{H}^*(G/P)\op{-SMod}^{\mathbb{Z}}_{\op{ev}} )
$$
between the category of stratified mixed Tate motives and the bounded homotopy category of complexes of even Soergel modules. 
\item Take $k=\mathbb{C}$ and $\mathscr{T}$ the category of modules over the  semisimplified Hodge cohomology. Then Hodge (hyper-)cohomology induces an   equivalence of categories 
$$
\DMT_{(B)}(G/P)_{w=0}\sirra 
\op{H}^\ast(G/P)\op{-SMod}_{\op{ev}}^{\mathbb{Z}}
$$
between the heart of the weight structure in (1) and the category of
even Soergel modules over the cohomology ring of $G/P$. This
equivalence extends to an equivalence 
$$
\DMT_{(B)}(G/P)\sirra
\op{Hot}^{\op{b}}(\op{H}^\ast(G/P)\op{-SMod}_{\op{ev}}^{\mathbb{Z}}).
$$
\end{enumerate}
\end{mainthm}

\subsection{Perverse t-structures}

The second block of results to be proved in this paper concerns a
perverse t-structure on the category of stratified mixed Tate
motives. While the existence of motivic t-structures is a very
difficult problem, there are some situations where the
Beilinson--Soul{\'e} vanishing conjectures and hence the existence of
a motivic t-structure on mixed Tate motives are
known. Alternatively, over $\mathbb{C}$, it is possible to work in a
category of motives with coefficients in the semisimplification of the
Hodge realization; in this case,  the category of mixed Tate motives over the
point is the derived category of graded vector spaces and therefore
has a natural ``motivic'' t-structure. In situations as above, we can
use the perverse formalism of \cite{BBD} to equip 
the category of stratified mixed Tate motives with a perverse
t-structure, for any perversity function
$p:\mathcal{S}\to\mathbb{Z}$. Its  heart is an abelian category 
$\PMT_{\mathcal{S}}(X)$ of perverse mixed Tate motives. 
 The following results are combinations of  the results of
\prettyref{sec:tstructure} and \prettyref{sec:cato}, more precisely
\prettyref{thm:perv}, \prettyref{thm:PrOoi}
and \prettyref{thm:cato}; everything is specialized to the two cases
of interest (related to $\ell$-adic resp. Hodge realizations). We suppress the underlying motivic triangulated category in our notation $\PMT_{\mathcal{S}}(X)$  for perverse motives, to underline that our methods give a uniform proof for both theorems below:

\begin{mainthm}
\label{thm:main2}
 Let $k$ be a finite field and $\mathscr{T}$ be the motivic triangulated category of (Beilinson or \'etale) motives with rational coefficients, and let $(X,\mathcal{S})$ be an affinely Whitney--Tate stratified $k$-variety in the  sense of \prettyref{defin:cellvar} and \prettyref{defin:WT}.
\begin{enumerate}
\item The category $\PMT_{\mathcal S}(X)$ has enough projectives
  and the tilting functor of \prettyref{prop:rvt} induces an
  equivalence of categories 
$$
\op{Der}^{\op{b}}(\PMT_{\mathcal S}(X))\sirra \DMT_{\mathcal S}(X)
$$
between the bounded derived category of the abelian category of perverse mixed Tate motives and the triangulated category of stratified mixed Tate motives.
\item If we consider motives with
$\mathbb Q_\ell$-coefficients where $\ell$ is a prime different from the
characteristic of $k$, the $\ell$-adic realization 
$$
 \PMT_{\mathcal S}(X;\DQ_\ell)\ra \op{Perv}_{\mathcal S}(X\times_k\bar
k;\DQ_\ell)
$$
is a degrading functor in the sense of \cite{BGSo}.  
\end{enumerate}
\end{mainthm}

\begin{mainthm}\label{thm:mainh}
Let $(X,\mathcal{S})$ be an affinely Whitney--Tate stratified
variety over $\mathbb{C}$  and $\mathscr{T}$ the motivic triangulated category of modules over the semisimplified Hodge realization. 
\begin{enumerate}
\item
The category $\PMT_{\mathcal S}(X)$ has enough projectives
  and the tilting functor of \prettyref{prop:rvt} induces an
  equivalence of categories 
$$
\op{Der}^{\op{b}}(\PMT_{\mathcal{S}}(X))
\sirra
\DMT_{\mathcal{S}}(X)
$$
between the bounded derived category of the abelian category of perverse mixed Tate motives and the triangulated category of stratified mixed Tate motives.
\item
Combining in the case of the full flag variety $G/B$ the Hodge realization with the algebraic Riemann--Hilbert correspondence  and Beilinson--Bernstein localization is a degrading functor   
$$
\PMT_{(B)}(G/B)\to\mathcal{O}_0.
$$
\end{enumerate}
\end{mainthm}

\begin{remark}
The algebraic Riemann--Hilbert correspondence of \cite{drew} allows to explicitly relate the category $\DMT_{\mathcal{S}}(X)$ in the Hodge situation to a suitable derived category of holonomic $\mathscr{D}$-modules on $X$. In particular, while in \prettyref{thm:main2} the link from the graded version to actual representations was rather weak,  the Hodge situation with its Riemann--Hilbert correspondence actually provides a direct relation between the above motivic graded categories $\PMT_{(B)}(G/B)$ with category $\mathcal{O}_0$, more precisely the category of finitely generated $\mathfrak g$-modules locally finite under a Borel subalgebra and annihilated by the central annihilator of the trivial one-dimensional representation, the latter realized via Beilinson--Bernstein localization as a category of $\mathscr{D}$-modules on the flag variety.
\end{remark}

\begin{remark}
This result recovers part of the results of
\cite{AK}. Eventually, our construction of a graded version
of category $\mathcal{O}$ also boils down to split some unwanted
extensions. However, using the motivic framework developed thus far
allows to shift all the technical difficulties into the notions of
motivic triangulated categories and representing sheaves of spectra
for enriched Weil cohomology theories. We hope this makes the actual
construction of graded versions of category $\mathcal{O}$  more transparent.  
\end{remark}

\subsection{Koszul duality remarks}

In Corollary \ref{bhu} we investigate the interaction of weights and
perversity in somewhat more detail. Suppose we are in the situation of \prettyref{thm:main1X}. Then for an affinely Whitney--Tate stratified variety $(X,\mathcal S)$  satisfying the pointwise purity conditions the above results put together provide equivalences of categories  
\begin{equation*}
  \op{Der}^{\op{b}} (\PMT_{\mathcal S} (X)) \stackrel{\approx}{\longrightarrow}
  \DMT_{\mathcal S}(X) \stackrel{\approx}{\longleftarrow}
  \op{Hot}^{\op{b}} (\DMT_{\mathcal S}(X)_{w=0}).
\end{equation*}
In this case, we sketch in \ref{Koszul}, why the corresponding 
category of perverse sheaves is governed by a Koszul ring.
A special case, in which all the above conditions are satisfied, is the case of a partial flag variety $X=G/P$ over $\DC$ with $\mathcal{S}=(B)$ the stratification by Borel orbits. In this case, the same arguments as in \cite{BGSo} exhibit $\PMT_{(B)}(G/P)$  as a graded version of the principal block of parabolic category $\mathcal{O}$, up to formally  adding a root of the Tate twist. Summing up, given a partial flag variety $X=G/P$ over  $\DC$   the above results provide equivalences of categories 
\begin{equation*}
  \op{Der}^{\op{b}} (\PMT_{(B)} (G/P)) \stackrel{\approx}{\longrightarrow}
  \DMT_{(B)}(G/P) \stackrel{\approx}{\longleftarrow}
  \op{Hot}^{\op{b}} (\DMT_{(B)}(G/P)_{w=0}).
\end{equation*}
The right hand side in turn is equivalent to $\op{Hot}^{\op{b}} (\op{H}^*(G/P)\op{-SModf}^\DZ_{\op{ev}})$ and can be identified as in  \cite{So-A}, up to formally adding a root of the Tate twist,  with the bounded derived category  of some graded version of some block of some category $\mathcal O$ for the Langlands dual Lie algebra, which is more or less singular depending on our parabolic. On the other hand, $\PMT_{(B)} (G/P)$ can be  identified as in \prettyref{thm:mainh},  up to formally adding a root of the Tate twist, with some graded version of a block of parabolic category $\mathcal O$. Putting  all this together, our above results allow to reconstruct the parabolic-singular duality of \cite{BGSo} in a slightly more concrete  way.  In particular, the Koszul self-duality for the principal block  of category $\mathcal O$ can be interpreted as   a completely canonical equivalence of triangulated $\DQ$-categories 
\begin{equation*}
K:  \DMT_{(B)}(G/B)\;\stackrel{\approx}{\longrightarrow}\; \DMT_{(B^\vee)}
 (G^\vee/B^\vee)
\end{equation*}
with the property $K(M(n))=(KM)(-n)[-2n]$ transforming indecomposable
injective perverse objects to simple perverse objects,
simple perverse objects to projective perverse objects, and perverse 
$\nabla$-sheaves to perverse $\Delta$-sheaves, by the way turning
their Weyl group parameters upside down. In \cite{BezYu} such an
equivalence is established with similar arguments in the setting of mixed $\ell$-adic sheaves. It would be very interesting to have a geometric
construction of such a functor. 

\subsection{Structure of the paper:}
We begin with a short recollection on triangulated categories of
motives in \prettyref{sec:cdbmot}, and a recollection on mixed Tate
motives in \prettyref{sec:mtm}. The Whitney--Tate condition and the
description of the category of stratified mixed Tate motives
is recalled in \prettyref{sec:dmt}, some more detailed discussion of
the Whitney--Tate condition is deferred to Appendix~\ref{sec:WTS}. In
\prettyref{sec:weights}, we 
explain the weight structure on stratified mixed Tate
motives. Pointwise purity and its relevance for the study of the heart
is discussed in  Sections \ref{sec:ptwise1} and \ref{sec:ptwise2}. In \prettyref{sec:ginzburg}, we reformulate
Ginzburg's full faithfulness result in the motivic setting. The latter
result is used in \prettyref{sec:TE} to prove a tilting result
identifying stratified mixed Tate motives  with the homotopy category of Soergel
modules. Some background on tilting can be found in
Appendix~\ref{sec:tilting}. In \prettyref{sec:tstructure}, we discuss the
perverse t-structure on stratified mixed Tate motives, and in
\prettyref{sec:cato} we show how $\ell$-adic and Hodge realization
functors of perverse mixed Tate motives provide a grading on category
$\mathcal{O}$.  

\subsection{Conventions:}
In a category $\mathcal{C}$, we denote by $\mathcal{C}(A,B)$ the set
of morphisms from $A$ to $B$. 
The symbol $\op{Hom}$ is reserved for ``inner hom''.
Homotopy categories are typically
denoted by $\op{Hot}$, derived categories by $\op{Der}$. For an
object $X$ of a category with a final object $*$, 
we denote by $\op{fin}=\op{fin}_X:X\to *$ the unique
morphism.  For an
$S$-scheme $X$, we denote in particular 
by $\op{fin}:X\to S$ the structure morphism. 
Most of the time, we work with the category $\op{Sch}/k$ of schemes which are separated and of finite type over a base field $k$. We occasionally might refer to those objects as varieties.

\subsection{Acknowledgements: }
We would like to thank Rahbar Virk, J\"org Wildeshaus and Jens Eberhard for  pointing out flaws and providing numerous helpful comments on a previous version. We are grateful to Fr{\'e}d{\'e}ric D{\'e}glise for pointing out the relevance of  Bradley Drew's thesis for applications to category
$\mathcal{O}$ in characteristic zero, and for explaining the results of the thesis to us. We also thank Bradley Drew for comments on the Hodge-part of this paper. Finally, we would like to thank the anonymous referees for careful reading and helpful suggestions.

\section{Triangulated categories of motives
  (d'apr{\`e}s Ayoub, Cisinski--D{\'e}glise,...)} 
\label{sec:cdbmot}

In this section, we provide a recollection of the construction and
properties of triangulated categories of motives  and the 
corresponding six-functor formalism. The general idea of motives and
the six functors as a formalization of cohomological properties of
algebraic varieties goes back to the development of \'etale 
cohomology by Grothendieck and his collaborators in the SGA
volumes. While the construction of an abelian category of motives
depends on difficult open conjectures, there are now reasonably good
triangulated categories of motives available. This  is based on
work of Voevodsky \cite{friedlander:suslin:voevodsky} who defined
triangulated categories of motives over a field. One possible approach for
establishing the existence and properties of the six functors in
motivic settings was proposed by Voevodsky and worked out in detail in
the thesis of Ayoub \cite{ayoub:thesis1,ayoub:thesis2}. Building on
this, constructions of 
triangulated categories of motives over rather general base schemes
together with constructions of the relating six functors were also given in
\cite{cisinski:deglise}. 

We will start the recollection with a discussion of the notion of
a motivic triangulated category, a framework for a motivic six-functor
formalism from \cite{cisinski:deglise}. Then we will recall two examples of motivic triangulated categories, namely \'etale motives \cite{ayoub:icm} and Beilinson motives \cite{cisinski:deglise}. After that, we discuss a third example of  motivic triangulated categories, namely the ones associated to enriched mixed Weil cohomology theories \cite{drew}. The results in our paper will be formulated for a general motivic triangulated category, but most of the representation-theoretic applications will additionally require the motivic triangulated category to satisfy the grading and weight conditions \ref{conditions}.

\subsection{Motivic triangulated categories}

As mentioned, there are several ways of encoding the properties of a
motivic six functor formalism. One possibility is the notion of
homotopical stable algebraic derivator of Ayoub \cite{ayoub:thesis1}, and
another is the notion of motivic triangulated categories of
Cisinski--D{\'e}glise \cite{cisinski:deglise}. We are going to list the
relevant properties of motivic triangulated categories which we will
need for the constructions in the  paper. For details, the reader is
referred to \cite{cisinski:deglise}.

\begin{definition}
Let $\mathscr{S}$ be a  category called the ``base category'' together with a class $\mathscr{P}$ of morphisms called ``$\mathscr{P}$-morphisms'', which is stable under composition and base change and contains all isomorphisms.
\begin{enumerate}
\item A  $2$-functor  $\mathscr{M}:\mathscr{S}^{\op{op}}\rightarrow\mathscr{C}at$ is called $\mathscr{P}$-fibred if for any morphism $p$ in $\mathscr{P}$ the functor $p^\ast$ has a left adjoint $p_\sharp$ and  for any cartesian square 
\begin{center}
  \begin{minipage}[c]{10cm}
    \xymatrix{
      Y \ar[r]^q \ar[d]_g & X \ar[d]^f \\
      T \ar[r]_p & S
    }
  \end{minipage}
\end{center}
with $p$ a $\mathscr{P}$-morphism the natural exchange transformation is an isomorphism $q_\sharp g^\ast\sira f^\ast p_\sharp$, cf. \cite[Definitions 1.1.1, 1.1.2 and 1.1.10]{cisinski:deglise}.  
\item Let $\mathscr{C}at^\otimes$ denote the 2-category of symmetric monoidal categories. A $\mathscr{P}$-fibred $2$-functor   $\mathscr{M}:\mathscr{S}^{\op{op}}\rightarrow\mathscr{C}at^\otimes$ is called monoidal if for any $\mathscr{P}$-morphism $f:T\rightarrow S$ and all $M$ and $N$ the natural exchange transformation is an isomorphism  $Ex(f_\sharp,\otimes_T):f_\sharp(M\otimes_T f^\ast N)\sira f_\sharp M\otimes_SN$,  cf. \cite[Definitions 1.1.21 and 1.1.27]{cisinski:deglise}.
\item Let $\mathscr{T}ri^\otimes$ denote the 2-category of triangulated monoidal categories. A $\mathscr{P}$-fibred monoidal $2$-functor $\mathscr{M}:\mathscr{S}^{\op{op}}\rightarrow\mathscr{T}ri^\otimes$ is called a $\mathscr{P}$-premotivic triangulated category if all pull-back functors $f^\ast$ admit triangulated right adjoints $f_\ast$ and all $M\otimes_S(-)$ admit right adjoints $\op{Hom}_S(M,-)$,   cf. \cite[Definition 1.4.2]{cisinski:deglise}. 
\end{enumerate}
\end{definition} 

\begin{remark}
From now on, we  will only consider the base category $\mathscr{S}=\mathscr{S}_k$ of separated schemes of finite type over some field $k$ with $\mathscr{P}$ the class of smooth morphisms of finite type. A $\mathscr{P}$-premotivic category will henceforth just be called a premotivic category over $k$ or, in case the ground field is fixed anyhow, a premotivic category.
\end{remark}


Next we recall from \cite[Section 2]{cisinski:deglise} further
properties that make a premotivic category motivic.

\begin{definition} Let $k$ be a fixed ground field.
\begin{enumerate}
\item A premotivic triangulated category $\mathscr{T}$ satisfies the \emph{homotopy property} if for any scheme $S\in\mathscr{S}$  the counit of the adjunction associated to the  projection $p:\Ao_S\rightarrow S$  is an isomorphism $1\sira  p_\ast p^\ast$. 
\item A premotivic triangulated category $\mathscr{T}$ satisfies the \emph{stability property} if for every $S\in\mathscr S$ and every smooth $S$-scheme $f:X\rightarrow S$ with section $s:S\rightarrow X$ in $\mathscr{S}$, the associated Thom transformation $f_\sharp s_\ast$ is an   equivalence of categories. 
\item A premotivic triangulated category $\mathscr{T}$ satisfies the \emph{localization property} if $\mathscr{T}(\emptyset)=0$ and for each closed immersion $i:Z\rightarrow S$ with open complement $j:U\rightarrow S$, the pair $(j^\ast,i^\ast)$ is conservative and the counit $i^\ast i_\ast\rightarrow 1$ is an isomorphism.
\item A premotivic triangulated category satisfies the \emph{adjoint property} if for any proper morphism $f$ in $\mathscr{S}$, the functor $f_\ast$ admits a right adjoint $f^!$.
\item A \emph{motivic triangulated category} over $k$ is a premotivic  triangulated category which satisfies the homotopy, stability, localization and adjoint properties, cf. \cite[Definition 2.4.45]{cisinski:deglise}.
\end{enumerate}
\end{definition}

\begin{remark}
Via the procedure in \cite[1.1.34]{cisinski:deglise} one can associate motives in $\mathscr{T}(S)$ to smooth morphisms of 
varieties $X\to S$. This uses the fact that $p^\ast:\mathscr{T}(S)\to\mathscr{T}(X)$ is required to have a left adjoint for smooth $p$, and one defines the motive $M_S(X)=p_\sharp(\mathbb{Q}_X)\in\mathscr{T}(S)$. Here $\mathbb{Q}_X$ denotes the tensor unit in $\mathscr{T}(X)$ 
which will also be denoted by $\underline{X}$ later on.
Sometimes we use the abbreviation $\mathscr T(X)=\mathscr T_X$ and often we use the notation   $\mathscr T_X(\mathcal F,\mathcal G)$ for  spaces of
morphisms in $\mathscr T(X)$. From the homotopy, stability and localization properties, one gets a computation of the motive of $\mathbb{P}^1_S$ as $M_S(\mathbb{P}^1)=\mathbb{Q}_S\oplus\mathbb{Q}_S(1)[2]$, where $\mathbb{Q}(1)$ is $\otimes$-invertible. This motive is called the \emph{Tate motive}, and tensoring with it is called \emph{Tate twist}. 
\end{remark}

For a motivic triangulated category $\mathscr T$, the following properties hold, as
proved in \cite{cisinski:deglise}. The list below is a variant of
the dix le{\c c}ons in \cite{hebert}, where we omitted those statements
that are contained in the definition above or are specific to
Beilinson motives.  

\begin{enumerate} 
\item For any morphism $f:Y\rightarrow X$ in $\mathscr{S}$, the adjunctions lead to natural isomorphisms 
$$
\op{Hom}_X(M,f_\ast N)\cong f_\ast \op{Hom}_Y(f^\ast M,N),
$$
cf. \cite[1.1.33]{cisinski:deglise}.
\item For any  morphism $f:Y\rightarrow X$ in $\mathscr{S}$, 
one can construct a further pair of adjoint functors, the {\bf exceptional 
  functors} 
$$
f_!:\mathscr T(Y)\leftrightarrows
\mathscr T(X):f^! 
$$
which fit together to form a covariant (resp. contravariant) 
$2$-functor  $f\mapsto f_!$ (resp. $f\mapsto f^!$). 
There exists a natural transformation $\alpha_f:f_!\rightarrow
f_\ast$ which is an isomorphism when $f$ is proper. Moreover,
$\alpha$ is a morphism of $2$-functors. 
\item For any cartesian square 
\begin{center}
  \begin{minipage}[c]{10cm}
    \xymatrix{
      X'\ar[r]^{g'} \ar[d]_{f'} & X\ar[d]^f \\
      Y' \ar[r]_g & Y
    }
  \end{minipage}
\end{center}
there exist natural isomorphisms of functors
$$
g^\ast f_!\stackrel{\sim}{\longrightarrow}f_!'g'^\ast,\qquad
g'_\ast f'^!\stackrel{\sim}{\longrightarrow}f^!g_\ast,
$$ 
cf. \cite[Theorem 2.2.14]{cisinski:deglise}
\item For any morphism  $f:Y\rightarrow X$ in
  $\mathscr{S}$, there exist natural isomorphisms 
$$
Ex(f_!^\ast,\otimes):(f_!K)\otimes_XL\stackrel{\sim}{\longrightarrow}
f_!(K\otimes_Y f^\ast L),
$$
$$
\op{Hom}_X(f_!L,K)\stackrel{\sim}{\longrightarrow}
f_\ast \op{Hom}_Y(L,f^!K),
$$
$$
f^!\op{Hom}_X(L,M)\stackrel{\sim}{\longrightarrow}
\op{Hom}_Y(f^\ast L,f^!M),
$$
cf. \cite[Theorem 2.2.14]{cisinski:deglise}.
\item For $f:X\to Y$ a smooth morphism of relative dimension $d$,
  there are canonical natural isomorphisms  
$$
\mathfrak{p}_f:f_\sharp\rightarrow f_!(d)[2d], \qquad
\mathfrak{p}_f^\prime:f^\ast\rightarrow f^!(-d)[-2d],
$$
cf. \cite[Theorem 2.4.50]{cisinski:deglise}.
\item An alternative formulation of the {\bf localization property},
  cf. \cite[Definition   2.3.2]{cisinski:deglise}: for $i:Z\rightarrow
  X$ a closed immersion with open complement $j:U\to 
X$, there are distinguished triangles of natural transformations
$$
j_!j^!\rightarrow 1\rightarrow i_\ast i^\ast \rightarrow j_!j^![1]
$$
$$
i_\ast i^!\to 1\to j_\ast j^\ast \to i_\ast i^! [1]
$$
where the first and second maps are the counits and units  of the
respective adjunctions, cf. \cite[Proposition 2.3.3, Theorem
2.2.14]{cisinski:deglise}. 

\item For any closed immersion $i:Z\rightarrow S$ of pure codimension
  $n$ between regular schemes in $\mathscr S$, 
the standard map $\op{M}_Z(Z)\rightarrow
    i^! \op{M}_S(S)(n)[2n]$ is an isomorphism, cf. \cite[Theorem
    14.4.1]{cisinski:deglise}. 
\item 
Define the subcategory of constructible objects $\mathscr{T}^c(S)\subset \mathscr{T}(S)$ to be the thick full subcategory generated by $\op{M}_S(X)(n)$ for $n\in\mathbb{Z}$ and $X\to S$ smooth. This subcategory coincides with the full subcategory of compact objects if motives of smooth varieties are compact, cf. \cite[Proposition 1.4.11]{cisinski:deglise}. Under some conditions satisfied in all cases we consider, the motives of smooth varieties are compact and the six functors preserve compact objects, cf. \cite[Section 4.2]{cisinski:deglise} and \cite[Theorem 15.2.1]{cisinski:deglise} for the  statement for Beilinson motives. 
\item  For $f:X\to S=\op{Spec} k$ a morphism in $\mathscr{S}$, the motive  $f^!(\op{M}_S(S))$ is a dualizing object, i.e., setting $D_X(M)=\op{Hom}(M,f^!(\op{M}_S(S)))$ the natural map $M\rightarrow D_X(D_X(M))$  is an isomorphism for all $M\in \mathscr{T}^c(X)$. For  all $M,  N\in \mathscr{T}^c(X)$, there is a canonical duality isomorphism  
$$
D_X(M\otimes D_X(N))\simeq \op{Hom}_X(M,N).
$$ 
Furthermore, for any morphism $f:Y\rightarrow X$ in $\mathscr{S}$ and any $M\in\mathscr{T}^c(X)$ and $N\in \mathscr{T}^c(Y)$, there are natural isomorphisms  
$$
D_Y(f^\ast(M))\simeq f^!(D_X(M)),\qquad 
f^\ast(D_X(M))\simeq D_Y(f^!(M))
$$
$$
D_X(f_!(N))\simeq f_\ast(D_Y(M)),\qquad
f_!(D_Y(N))\simeq D_X(f_\ast(N)),
$$
cf. \cite[Theorem 15.2.4]{cisinski:deglise}.
\end{enumerate}

\subsection{Examples: \'etale motives, Beilinson motives}
\label{CBM} 

Next, we recall two instances of motivic triangulated categories,
namely \'etale motives \cite{ayoub:icm} and Beilinson motives
\cite{cisinski:deglise}. 
For \'etale motives, the construction of the categories
is carried out in detail in \cite[Section 4.5]{ayoub:thesis2}, and
an overview of the construction is given in \cite[Section
2.3]{ayoub:icm}. Beilinson motives are constructed in \cite[Section
14]{cisinski:deglise}. In the case of rational coefficients, which is the only relevant case for our work, both constructions turn out to lead to the same result and the reader can choose the construction they prefer.

Let $S$ be a separated scheme of finite type over a field $k$. The category ${\op{Sm}}/S$  of smooth schemes of finite type over $S$ admits among others  the Nisnevich topology and the \'etale topology. For $\tau\in\{\op{Nis},\et\}$, we denote by $\op{Sh}_{\tau}({\op{Sm}}/S,\mathbb{Q})$ the category of $\tau$-sheaves of $\mathbb{Q}$-vector spaces on $\op{Sm}/S$, cf. \cite[Example 5.1.4]{cisinski:deglise}. There is a model structure on the category of unbounded complexes in $\op{Sh}_{\tau}({\op{Sm}}/S,\mathbb{Q})$ whose weak equivalences are the quasi-isomorphisms, its homotopy category is the derived category $\op{Der}(\op{Sh}_{\tau}({\op{Sm}}/S,\mathbb{Q}))$. For $X\in{\op{Sm}}/S$ a smooth $S$-scheme, $\mathbb{Q}(X)$ denotes the ``representable'' sheaf associating to $U\in{\op{Sm}}/S$ the $\mathbb{Q}$-vector space freely generated by $\op{Sch}_S(U,X)$.  
One can then use a Bousfield localization (on the model category level) or a Verdier quotient (on the derived category level) to enforce $\Ao$-invariance, i.e., to turn the natural projection $\mathbb{Q}(X\times\Ao)\to \mathbb{Q}(X)$ into a quasi-isomorphism for any smooth $S$-scheme $X$, cf. \cite[5.2.b]{cisinski:deglise}.  The result is  the {\bf effective  $\Ao$-derived category}, denoted by $\op{Der}_{\Ao}^{\op{eff}}(\op{Sh}_{\tau}({\op{Sm}}/S,\mathbb{Q}))$,  cf. \cite[Example 5.2.17]{cisinski:deglise}. 

The monomorphism $1:S\rightarrow\mathbb{G}_{\op{m},S}$ in the category ${\op{Sm}}/S$ gives rise to a morphism $1:\DQ(S)\rightarrow\DQ(\mathbb{G}_{\op{m},S})$  of representable sheaves, viewed as complexes concentrated in degree  $0$. The cone of this morphism in $\op{Der}_{\Ao}^{\op{eff}}(\op{Sh}_{\tau}({\op{Sm}}/S,\mathbb{Q}))$ is  called suspended Tate $S$-premotive $\mathbb{Q}_S(1)[1]$.
 One can then use the formalism of symmetric spectra, cf. \cite{hovey}, to invert tensoring with the suspended Tate $S$-premotive, cf. \cite[Section 5.3]{cisinski:deglise}. The homotopy category of the corresponding model structure on symmetric spectra in the effective $\Ao$-derived category is called the {\bf stable $\Ao$-derived category} $\op{Der}_{\Ao,\tau}(S,\mathbb{Q})$, cf. \cite[Example 5.3.31]{cisinski:deglise}. 


At this point, for $\tau=\et$, we can define the category of \'etale
motives with rational coefficients to be the \'etale stable
$\Ao$-derived category with rational coefficients, $\op{Der}_{\Ao,\et}(S,\mathbb{Q})$.

On the other hand, for $\tau=\op{Nis}$,  Beilinson motives are
constructed as a  category of modules over the $0$-th graded part of
rational K-theory: 
following  the work of Voevodsky, Riou and Panin--Pimenov--R\"ondigs, 
there exists for each scheme $S$ a spectrum $\op{KGL}_S$ representing
Weibel's homotopy invariant K-theory in the stable homotopy category
$\op{SH}(S)$. With rational coefficients, the ring spectrum
$\op{KGL}_{\mathbb{Q},S}$ decomposes as a direct sum of Adams eigenspaces
$\op{KGL}_S^{(i)}$. The zeroth eigenspace $\op{KGL}_S^{(0)}$ is called the
{\bf Beilinson motivic cohomology spectrum} $\op{H}_{\Be,S}$,
cf. \cite[Definition 14.1.2]{cisinski:deglise}.  
The category  $\Bmot(S)$ of Beilinson motives over $S$ is then defined
to be the Verdier quotient of the $\Ao$-derived category
$\op{Der}_{\Ao,\op{Nis}}(S,\mathbb{Q})$ by the subcategory of $\op{H}_{\Be}$-acyclic objects.  Alternatively (glossing over the difficulties making $\op{H}_{\Be,S}$ a strict commutative ring spectrum), one can construct $\Bmot(S)$ as homotopy category of a model structure on a category of $\op{H}_{\Be,S}$-modules. For $X\in{\op{Sm}}/S$ a smooth $S$-scheme, the image of $\mathbb{Q}(X)$ in $\Bmot(S)$ is defined to be the motive $\op{M}_S(X)$ of $X$. 

Finally, we need to explain how morphisms between motives as above are
computed. For our applications, we will only need to compute morphisms between Tate motives. Working over a base field $k$ and with rational coefficients, the result is easy enough to state, cf. \cite[Theorem 4.12]{ayoub:icm} for the \'etale case and  \cite[Corollary 14.2.14]{cisinski:deglise} for the Beilinson case:
$$
\mathbf{DA}^{\et}_{\op{Spec}(k)}(\mathbb{Q},\mathbb{Q}(p)[q])\cong
\Bmot(\mathbb{Q},\mathbb{Q}(p)[q])\cong
\op{gr}_\gamma^p\op{K}_{2p-q}(k)_{\mathbb{Q}}.
$$
This means that in both cases the morphisms between Tate motives are given in terms of graded pieces of the $\gamma$-filtration on rational algebraic K-groups. The basic reason for this coincidence is the fact that rationally algebraic K-theory and \'etale K-theory are isomorphic. In particular, for mixed Tate motives with rational coefficients  discussed later, it does not matter which of the above categories we are working in. 

\subsection{Realization functors and mixed Weil cohomology theories}
\label{sec:realization}

Next, we discuss realization functors on the category of Beilinson
motives, cf. \cite[Section 17]{cisinski:deglise}. 

Fix a coefficient field $\mathbb{K}$ of characteristic $0$. For a
sheaf $E$ of commutative differential graded $\mathbb{K}$-algebras on
$\op{Sm}/k$, there is an associated cohomology theory
$$
\op{H}^n(X,E)
:=\op{Der}_{\Ao}^{\op{eff}}(\op{Sh}_{\op{Nis}}(\op{Sm}/k,\mathbb{Q}))
(\mathbb{Q}(X),E[n]).
$$
In the above, $X\in \op{Sm}/k$ is a smooth scheme.
This cohomology theory is called a {\bf mixed Weil cohomology
  theory}, if  it satisfies the following axioms,
cf. \cite[17.2.1]{cisinski:deglise}:
\begin{enumerate}
\item $\op{H}^0(\op{Spec}k,E)\cong \mathbb{K}$ and $\op{H}^i(\op{Spec}
  k,E)\cong 0$ for $i\neq 0$. 
\item $\dim_{\mathbb{K}}\op{H}^i(\mathbb{G}_{\op{m}},E)=\left\{\begin{array}{ll}
1&i=0 \textrm{ or }i=1;\\0&\textrm{otherwise.}
\end{array}\right.$
\item For any two smooth $k$-schemes $X$ and $Y$, the K\"unneth
  formula holds:
$$
\bigoplus_{p+q=n} \op{H}^p(X,E)\otimes_{\mathbb{K}} \op{H}^q(Y,E)\cong
\op{H}^n(X\times_k Y,E).
$$
\end{enumerate}
By \cite[Proposition 17.2.4]{cisinski:deglise}, any mixed Weil
cohomology theory $E$ is representable by a commutative ring spectrum
$\mathcal{E}$ in $\Bmot(k)$. In \cite[17.2.5]{cisinski:deglise},
realization functors on the category of Beilinson motives are defined
by considering the homotopy category of $\mathcal{E}$-modules over $X$
and taking the realization functor to be 
$$
\Bmot(X)\rightarrow \op{Der}(X,\mathcal{E}): M\mapsto
\mathcal{E}_X\otimes_X^L M.
$$ 
In the above, the category $\op{Der}(X,\mathcal{E})$ is the homotopy
category of a model structure on the category of $\mathcal{E}$-modules
in $\Bmot(X)$. These realization functors preserve compact objects,
cf. \cite[17.2.18]{cisinski:deglise}, hence we obtain realization
functors 
$$
\Bmotc(X)\rightarrow \op{Der}_c(X,\mathcal{E}): M\mapsto
\mathcal{E}_X\otimes_X^L M. 
$$
Both these realization functors commute with the six functor
formalism. Moreover, for any field extension $L/k$, there is an
equivalence of symmetric monoidal triangulated categories
$$
\op{Der}(L,\mathcal{E})\cong \op{Der}(\mathbb{K}\textrm{-mod})
$$
between the $\mathcal{E}$-modules over $L$ and the derived category of
$\mathbb{K}$-modules. This equivalence restricts to an equivalence
$\op{Der}_c(L,\mathcal{E})\cong
\op{Der}^{\op{b}}(\mathbb{K}\textrm{-modf})$ between the 
compact $\mathcal{E}$-modules over $L$ and the bounded derived
category of finitely generated $\mathbb{K}$-modules.

We list some examples of mixed Weil cohomology theories to which the
above results can be applied, cf. \cite[Section 3]{cd:weil}:

\begin{enumerate}
\item Algebraic de Rham cohomology is a mixed Weil cohomology with
  associated commutative ring spectrum $\mathcal{E}_{dR}$,
  cf. \cite[Section 3.1]{cd:weil}. 
\item Rigid cohomology is a mixed Weil cohomology theory with
  associated commutative ring spectrum $\mathcal{E}_{rig}$,
  cf. \cite[Section 3.2]{cd:weil}. 
\item $\ell$-adic cohomology is a mixed Weil cohomology theory with
  associated commutative ring spectrum $\mathcal{E}_{\op{et},\ell}$,
  cf. \cite[Section 3.3]{cd:weil}. 
\end{enumerate}

In particular, the $\ell$-adic and de Rham realization functors will
be relevant for our discussion.

\subsection{Enriched mixed Weil cohomology theories}
There is one other example of motivic triangulated category which we will consider. It arises from an enriched refinement of the mixed Weil cohomology theories discussed above. 
The existence of such motivic triangulated categories and their properties were worked out in the thesis of Drew, cf. \cite{drew}. 

Recall from \cite[Definition 2.1.1]{drew} the following definition of
a mixed Weil cohomology theory enriched in a Tannakian category. 

\begin{definition} 
Let $S$ be a noetherian scheme of finite Krull dimension, 
and let $\mathcal{T}_0$ be a Tannakian category of finite Ext-dimension, 
and denote $\mathcal{T}=\op{Ind-}\mathcal{T}_0$. A mixed Weil
cohomology theory enriched in $\mathcal{T}$ is a presheaf $\op{E}_S$ of
commutative differential graded algebras in $\mathcal{T}$ on the
category of smooth affine $S$-schemes, satisfying the following
axioms:
\begin{enumerate}[(W1)]
\item descent for Nisnevich hypercoverings,
\item $\mathbb{A}^1$-invariance,
\item normalization, i.e., $\op{E}_S(S)$ is contractible,
\item for $\sigma_1:S\mapsto \mathbb{G}_{\op{m},S}$ the unit section,
  the object
  $\mathbb{Q}_{\mathcal{T}}(-1):=\ker(\op{E}_S(\sigma_1)[1])$ belongs
  to the heart of the natural t-structure of $\op{Der}(\mathcal{T})$
  and induces an autoequivalence
  $\mathbb{Q}_{\mathcal{T}}(-1)\otimes^{\op{L}}_{\mathcal{T}}(-)$ of
  $\op{Der}(\mathcal{T})$. 
\item K\"unneth formula, i.e., for any smooth affine schemes $X,Y$
  over $S$, the canonical morphism
  $\op{E}_S(X)\otimes^{\op{L}}_{\mathcal{T}}\op{E}_S(Y)\to\op{E}_S(X\times_S
  Y)$ is a weak equivalence. 
\end{enumerate}
\end{definition}

\begin{example}[\cite{drew}, Theorem 2.1.8]
Associating to a smooth $\mathbb{C}$-scheme the singular cohomology of
its associated complex manifold, equipped with its polarisable mixed
Hodge structure, yields a mixed Weil cohomology 
theory with coefficients in  $\mathcal{MHS}^{\op{pol}}_{\mathbb{Q}}$.
\end{example}

\begin{proposition}[\cite{drew}, Theorem 2.1.4]
Let $S$ be a noetherian scheme of finite Krull dimension, and let
$\mathcal{T}_0$ be a Tannakian category of finite Ext-dimension, 
and denote $\mathcal{T}=\op{Ind-}\mathcal{T}_0$. For a mixed Weil
cohomology theory $\op{E}_S$ enriched in $\mathcal{T}$, there exists a
commutative ring spectrum $\mathcal{E}_S$ in the category
$\mathcal{SH}(S,\mathcal{T})$ of symmetric
$\mathbb{Q}_{\mathcal{T}}(1)$-spectra over $S$ with values in
$\mathcal{T}$-complexes which represents $\op{E}_S$.
\end{proposition}

\begin{proposition}[\cite{drew}, Proposition 2.2.1]
The assignment $X\mapsto \op{Mod}(\mathcal{E}_X)$ extends to a
monoidal motivic triangulated category $X\mapsto
\op{Der}(\mathcal{E}_X)$. In 
particular, the full six functor formalism (including the duality
statements for the compact objects) applies to $\op{Der}(\mathcal{E}_X)$.  
\end{proposition}

The particular enriched mixed Weil cohomology theory relevant for us
is a simplification of the abovementioned Hodge realization, which we
want to discuss now. Before, we shortly recall some statements
concerning real mixed Hodge structure from \cite{deligne}:

\begin{definition}
Let $V$ be a finite-dimensional $\mathbb{R}$-vector space. A real mixed
Hodge structure on $V$ is given by
\begin{enumerate}
\item a finite ascending \emph{weight filtration} $W^{\leq n}$ on
  $V$,
\item a finite descending \emph{Hodge filtration} $F^{\geq p}$ on
  $V\otimes_{\mathbb{R}}\mathbb{C}$,
\end{enumerate}
such that for $p+q\neq n$ and $\overline{F}^{\geq p}$ the conjugate
filtration to $F^{\geq p}$, we have 
$$
\op{Gr}^p_F \op{Gr}^q_{\overline{F}} \op{Gr}^n_W(V)=0.
$$
\end{definition}

\begin{proposition}
The category  of real mixed Hodge structures is an
abelian rigid tensor category. The functor sending a mixed Hodge
structure $M$ to its underlying real vector space is a fiber functor
making the category of real mixed Hodge structures a neutral Tannakian
category.  
\end{proposition}

\begin{proposition}
The functor 
$$
\op{Gr}^W:\mathcal{MHS}_{\mathbb{Q}}^{\op{pol}}\to
\mathcal{HS}_{\mathbb{C}}^{\op{pol},\mathbb{Z}}:
A\mapsto \bigoplus_{n\in\DZ}\op{Gr}^W_n(A\otimes_{\mathbb{Q}}\mathbb{C}),  
$$
which sends a mixed Hodge structure to the Hodge structure given by the direct sum of the subquotients of the weight filtration is an exact tensor functor. 
\end{proposition}

Assume $\Phi:\mathcal{T}_0\to \mathcal{T}'_0$ is an exact tensor
functor of Tannakian categories, and $\op{E}$ is a mixed Weil
cohomology theory with values in  $\mathcal{T}_0$. By \cite[Lemma
2.1.3]{drew},  $\Phi\circ \op{E}$ is a mixed Weil cohomology theory
with values in $\mathcal{T}'_0$. We apply this to
$\op{Gr}^W:\mathcal{MHS}_{\mathbb{Q}}^{\op{pol}}\to
\mathcal{HS}_{\mathbb{C}}^{\op{pol},\mathbb{Z}}$ and the refined Betti
cohomology $\op{E}_{\op{Hodge}}$ of \cite[Theorem 2.1.8]{drew}. 
We get a mixed Weil cohomology theory $\op{E}_{\op{GrH}}$ with
coefficients in graded pure Hodge structures. We get an
associated motivic category $\op{Der}(\mathcal{E}_{\op{GrH}})$ with a
full six-functor formalism, weight structures and all, by the
above-cited results of \cite{drew}. Moreover, by \cite[Theorem
2.2.7]{drew}, restricting to mixed Tate objects, there is an equivalence
$$\DMT(\mathcal{E}_{\op{GrH}})(\op{Spec}\mathbb{C})\cong
\op{Der}^{\op{b}}(\op{Modf}^{\mathbb{Z}}(\mathbb{C})),$$
and this equivalence respects t-structures, weight structures and
compact objects.  
Here $\op{Modf}^{\mathbb{Z}}(\mathbb{C})$ denotes the category of finitely generated $\DZ$-graded complex vector spaces.

\subsection{Weight structures}

Finally, we have to discuss weight structures on categories of
motives. We first recall, for the  reader's convenience, the
definition of weight structures from 
\cite[Definition 1.1.1]{bondarko:ktheory}. Note, however, that our
sign convention for the weight is opposite to the one of loc.cit. We
follow the sign convention used in most other works on weight
structures, such as \cite{wildeshaus:intermediate} and \cite{hebert}. 

\begin{definition}
\label{defin:wtstruct}
Let $\mathcal{C}$ be a triangulated category. A {\bf weight structure 
  on $\mathcal{C}$} is a pair $w=(\mathcal{C}_{w\leq
  0},\mathcal{C}_{w\geq 0})$ of full subcategories of $\mathcal{C}$
such that with the notations
$
\mathcal{C}_{w\leq n}:=\mathcal{C}_{w\leq 0}[n] $ and $
\mathcal{C}_{w\geq n}:=\mathcal{C}_{w\geq 0}[n]
$ the following conditions are satisfied: 
\begin{enumerate}
\item the categories $\mathcal{C}_{w\leq 0}$ and $\mathcal{C}_{w\geq 0}$  
are closed under taking direct summands;
\item $\mathcal{C}_{w\leq 0}\subset \mathcal{C}_{w\leq 1}$ and
  $\mathcal{C}_{w\geq 1}\subset \mathcal{C}_{w\geq 0}$;
\item for any pair of objects $X\in\mathcal{C}_{w\leq 0}$, $Y\in
  \mathcal{C}_{w\geq 1}$, we have $
\mathcal{C}(X,Y)=0$; 
\item for any object $X\in\mathcal{C}$ there is a distinguished
  triangle 
$
A\to X\to B\to A[1]
$
with $A\in\mathcal{C}_{w\leq 0}$ and $B\in\mathcal{C}_{w\geq 1}$. 
\end{enumerate}
The full subcategory $\mathcal{C}_{w=0}=\mathcal{C}_{w\leq
  0}\cap\mathcal{C}_{w\geq 0}$ is called the {\bf heart of the weight
  structure $w$}. 
\end{definition}

H{\'e}bert has constructed weight structures on the categories of
Beilinson motives. The result is the following, cf. \cite[Theorems
3.3 and 3.8]{hebert}:

\begin{theorem}
\label{thm:hebert}
Let $k$ be a field. For any separated scheme $X$ of finite type over
$k$, there is a canonical weight structure $w$ on $\Bmotc(X)$. The
family of these weight structures on $\Bmotc$ is characterized uniquely
by the following  three properties: 
\begin{enumerate}
\item if $X$ is regular, then $\mathbb{Q}_X(n)[2n]\in\Bmotc(X)_{w=0}$
  for all $n\in\mathbb{Z}$; 
\item for any separated finite type morphism $f:X\to Y$, the functors
  $f^\ast$, $f_!$ (and $f_\sharp$ for $f$ smooth) are $w$-left exact,
  i.e., they preserve non-positivity of weights;
\item for any separated finite type morphism $f:X\to Y$, the functors
  $f_\ast$, $f^!$ (and $f^\ast$ for $f$ smooth) are $w$-right exact,
  i.e., they preserve non-negativity of weights. 
\end{enumerate}
\end{theorem}

By \cite[Theorem 2.3.2-2.3.4]{drew}, there are
$\mathcal{E}_X$-analogues of H{\'e}bert's theorems on weight
structures for Beilinson motives provided the following axiom is
satisfied:
\begin{enumerate}[(W6)]
\item for all smooth affine schemes $X$ over the base and 
  $r,s\in\mathbb{Z}$ with $2r<s$ we have
$$
\op{Der}(\mathcal{T})(\mathbb{Q}_{\mathcal{T}},\op{E}_S(X)(r)[s])=0,
$$
\end{enumerate}

Whenever this axiom is satisfied, then there is,  for each $S$-scheme
$X$, a weight structure on $\op{Der}(\mathcal{E}_X)$ whose heart is
generated by $\mathcal{E}_X$-motives of $Y$ for $f:Y\to X$ projective
and $Y$ regular.   
This in particular applies to the motivic categories
associated to the Hodge realization with values in
$\mathcal{MHS}^{\op{pol}}_{\mathbb{Q}}$, cf. the remark on p.8 of
\cite{drew}. Since $\op{E}_{\op{GrH}}$ arises from taking the
associated graded of the Hodge realization, validity of axiom (W6) for
$\op{E}_{\op{Hodge}}$ implies the validity of axiom (W6) for
$\op{E}_{\op{GrH}}$ so that the categories of
$\op{E}_{\op{GrH}}$-modules will have weight structures satisfying the
statement of \prettyref{thm:hebert}.

\section{Mixed Tate motives}
\label{sec:mtm}

In this section, we discuss triangulated categories of mixed Tate
motives as well as weight and t-structures on them. We mainly follow
\cite{levine:nato}, \cite{levine:tata} and \cite{wildeshaus:cras1} for
the t-structures and \cite{hebert},\cite{wildeshaus:at} for the
weight structures. While triangulated categories 
of mixed Tate motives and weight structures on them can be defined in
a rather general setup, the existence of a 
non-degenerate t-structure and a corresponding abelian category of
mixed Tate motives depends on the Beilinson--Soul{\'e} vanishing
conjectures. For our purposes, this will suffice as we are mostly
interested in the existence of mixed Tate motives over smooth
varieties with an $\mathbb{A}^n$-filtration over fields where the
Beilinson--Soul{\'e} conjectures are known to hold.

\subsection{Triangulated mixed Tate motives}
Fix a motivic triangulated category $\mathscr{T}$. 
Recall that for a scheme $S$, the suspended Tate motive
$\mathbb{Q}_S(1)[1]$  is defined as the cone of $\op{M}_S(S)\to
\op{M}_S(\mathbb{G}_{\op{m},S})$ in $\mathscr{T}(S)$,
cf. \prettyref{sec:cdbmot} above or
\cite[5.3.15]{cisinski:deglise}. Then one sets 
$\mathbb{Q}_S(n)=\mathbb{Q}_S(1)^{\otimes n}$. The following is
\cite[Definition 3.14]{levine:tata}. 

\begin{definition}
For each smooth $k$-scheme $S$, we define the {\bf triangulated
  category of mixed Tate motives over $S$}, denoted by
$\DMT(S)=\DMT(\mathscr T,S)$, to be the strictly full triangulated subcategory
of $\mathscr{T}(S)$ generated by the objects $\mathbb{Q}_S(n)$.  
\end{definition}
In the following, whenever we consider a category of mixed Tate
motives over a scheme $S$, this scheme will always be smooth. 
There are some direct consequences of this definition,
cf. \cite[Proposition 3.15]{levine:tata}. 
\begin{proposition}
\label{prop:dmttensor}
The category $\DMT(S)$ is a tensor triangulated
category. Its objects are compact, i.e., there is an
inclusion $\DMT(S)\subset\mathscr{T}^c(S)$.
\end{proposition}

\begin{proof}
The first follows from $\mathbb{Q}_S(n)\otimes\mathbb{Q}_S(m)\cong
\mathbb{Q}_S(n+m)$. The second follows from \cite[Section
4.2]{cisinski:deglise} and the definition of $\mathbb{Q}_S(1)$ which
only uses the smooth $S$-schemes $S$ and $\mathbb{G}_{\op{m},S}$. 
\end{proof}

\begin{remark}
  If $k$ is a field, then by \cite[Theorem 2.5]{wildeshaus:at} the restriction
  of H\'ebert's weight structure $w$ on $\Bmotc(k)$ is a weight structure on
  $\DMT(k)$.\label{WDM} 
\end{remark}

\subsection{Weight t-structure on mixed Tate motives 
(d'apr{\`e}s Levine)}

In the following, we recall \cite[Definition 3.16 and Theorem
3.19]{levine:tata}. It provides an approach to defining weights for
mixed Tate motives different from the weight structures discussed
above. Levine's approach uses a suitable t-structure on the
triangulated category of mixed Tate motives, and produces different
weights related to the weights defined above via d{\'e}calage.

In the following, we consider categories $\DMT(S)$ which are
triangulated categories of Tate type, in the sense of \cite[Definition
1.1]{levine:nato}. This means, in addition to \prettyref{prop:dmttensor}, that the following conditions on morphisms in $\mathscr{T}_S$ are satisfied: 
\begin{enumerate}
\item $\mathscr{T}_S(\mathbb{Q}_S(n)[a],\mathbb{Q}_S(m)[b])=0$ if $n>m$.
\item $\mathscr{T}_S(\mathbb{Q}_S(n)[a],\mathbb{Q}_S(n)[b])=0$ if $a\neq
  b$.
\item
  $\mathscr{T}_S(\mathbb{Q}_S(n),\mathbb{Q}_S(n))=\mathbb{Q}\cdot\op{id}$.
\end{enumerate}
We will only need the weight t-structure to define the t-structure
on mixed Tate motives below. Therefore, it suffices to say that these
conditions are in particular satisfied whenever $S$ satisfies the
Beilinson--Soul{\'e} vanishing conditions. However, the conditions hold
in greater generality: if $S$ is the spectrum of a field, vanishing
results in K-theory show that $\DMT(S)$ is of Tate type,
cf. \cite[Theorem 4.1]{levine:nato}.

\begin{definition}
Denote by $W_n\DMT(S)$ the strictly full triangulated
subcategory of $\DMT(S)$ generated by the Tate motives
$\mathbb{Q}_{S}(-a)$ with $a\leq n$. 
Denote by $W_{[n,m]}\DMT(S)$ the strictly full triangulated
subcategory of $\DMT(S)$ generated by the Tate motives
$\mathbb{Q}(-a)$ with $n\leq a\leq m$. 
Denote by $W^{>n}\DMT(S)$ the strictly full triangulated
subcategory of $\DMT(S)$ generated by the Tate motives
$\mathbb{Q}(-a)$ with $a>n$. 
\end{definition}

The following recalls \cite[Theorem 3.19]{levine:tata}.
\begin{proposition}
\label{prop:wt}
Assume $\DMT(S)$ is a tensor triangulated category of Tate
type. Then the following statements are true:
\begin{enumerate}
\item $(W_n\DMT(S),W^{>{n-1}}\DMT(S))$ is a t-structure on  $\DMT(S)$ with
  heart $W_{[0,0]}\DMT(S)$. 
\item The truncation functors
$$
W_n:\DMT(S)\rightarrow W_n\DMT(S), 
W^{>n}:\DMT(S)\rightarrow W^{>n}\DMT(S)
$$
are exact, $W_n$ is right adjoint to the corresponding inclusion and
$W^{>n}$ is left adjoint to the corresponding inclusion. 
\item For each $n<m$ there is an exact functor
  $$
  W_{[n+1,m]}:\DMT(S)\rightarrow W_{[n+1,m]}\DMT(S)
  $$ 
  and a natural distinguished triangle
  $$
  W_n\rightarrow W_m\rightarrow W_{[n+1,m]}\rightarrow W_n[1].
  $$
\item 
$\DMT(S)=\bigcup_{n\in\mathbb{Z}}W_n\DMT(S)=
\bigcup_{n\in\mathbb{Z}}W^{>n}\DMT(S)$.
\end{enumerate}
\end{proposition}

We denote by $\op{gr}^W_n:\DMT(S)\to
W_{[n,n]}\DMT(S)$ the corresponding composition of truncation
functors, it assigns to a mixed Tate motive the $n$-th subquotient of
the weight filtration. 

If $\DMT(S)$ is a category of Tate type, then the category 
$W_{[n,n]}\DMT(S)$ can be identified with the derived category
$\op{Der}^{\op{b}}(\mathbb{Q}\op{-modf})$ of finite-dimensional
$\mathbb{Q}$-vector spaces.

\subsection{t-structure on mixed Tate motives (d'apr{\`e}s Levine)} 

Now we recall the existence of abelian categories of mixed Tate
motives under the assumption of the Beilinson--Soul{\'e} vanishing
conjectures, cf. \cite[Definition 3.21, Theorem 3.22]{levine:tata}, 
cf. also the field case \cite[Theorem 1.4, Proposition 2.1, Theorem
4.2]{levine:nato}. 

\begin{definition}
\label{defin:vanish}
Fix a motivic triangulated category $\mathscr{T}$. 
We say that a separated smooth finite type $k$-scheme $S$
satisfies {\bf Beilinson--Soul{\'e} vanishing} for $\mathscr{T}$ if for
$m<0$,  we have 
$$
{\mathscr{T}_k}(\op{M}_k(S),\mathbb{Q}_k(n)[m])=0. 
$$
As mentioned above, we identify $W_{[n,n]}\DMT(\mathscr{T},S)$ with
$\op{Der}^{\op{b}}(\mathbb{Q}\op{-modf})$, and this allows to define
for each mixed Tate motive $M\in \DMT(\mathscr{T},S)$ the
$\mathbb{Q}$-vector space $\op{H}^m(\op{gr}^W_nM)$, cf. \cite[Remark
3.20]{levine:tata}.  

Let $\DMT(\mathscr{T},S)^{\leq 0}$ be the full subcategory of those
$M\in \DMT(\mathscr{T},S)$ such that  
$$
\op{H}^m(\op{gr}^W_nM)=0 \textrm{ for all }m>0 \textrm{ and
}n\in\mathbb{Z}. 
$$

Let $\DMT(\mathscr{T},S)^{\geq 0}$ be the full subcategory of those
$M\in \DMT(\mathscr{T},S)$ such that 
$$
\op{H}^m(\op{gr}^W_nM)=0 \textrm{ for all }m<0 \textrm{ and
}n\in\mathbb{Z}. 
$$

Finally, we set $\op{MT}(\mathscr{T},S)=\DMT(\mathscr{T},S)^{\leq 0}\cap
\DMT(\mathscr{T},S)^{\geq 0}$.
\end{definition}

\begin{theorem}
\label{thm:levbs}
Suppose the smooth scheme $S$ satisfies the Beilinson--Soul{\'e}
vanishing conjectures for $\mathscr{T}$.  
\begin{enumerate}
\item $(\DMT(\mathscr{T},S)^{\leq 0},\DMT(\mathscr{T},S)^{\geq 0})$ is a
  non-degenerate t-structure on the category
  $\DMT(\mathscr{T},S)$ with heart $\op{MT}(\mathscr{T},S)$
  containing the Tate motives $\mathbb{Q}_S(n)$, 
  $n\in\mathbb{Z}$. 
\item The category $\op{MT}(\mathscr{T},S)$ is a rigid $\mathbb{Q}$-linear
  abelian tensor category. 
\end{enumerate}
\end{theorem}

\begin{remark}
Recall that for $\mathscr{T}$ given by \'etale or Beilinson motives
the homomorphisms in $\DMT(k)$ can be computed from rational
K-theory as 
$$
\DMT(k)(\mathbb{Q}(n),\mathbb{Q}(n+q)[p])\cong
\op{K}_{2q-p}(k)^{(q)}. 
$$
In the case of global fields and finite fields, there is also a
precise relation between Ext-groups in the abelian category of mixed
Tate motives and rational K-theory. More precisely, there are natural
isomorphisms, cf. \cite[Corollary 4.3]{levine:nato}: 
$$
\op{Ext}^p_{\op{MT}(k)}(M,N)\sira\DMT(k)(M,N[p]).
$$
In particular, the vanishing of rational K-theory for finite fields
and global function fields implies that for such $k$, there are no
extensions between objects in $\op{MT}(k)$. 
\end{remark}

\begin{proposition}
\label{prop:vanish}
Let $k$ be a field satisfying the Beilinson--Soul{\'e} vanishing
conjectures for $\mathscr{T}$. Assume $S$ is smooth and $\op{M}_k(S)$
is in $\DMT(k)$. Then $S$ also satisfies the Beilinson--Soul{\'e}
vanishing conjectures for $\mathscr{T}$. 
\end{proposition}

\begin{proof}
As in the proof of \prettyref{prop:wt}, we have (together with
Beilinson--Soul{\'e} for the base field)
$$
\op{Hom}_{\mathscr{T}(k)}(\mathbb{Q}_k(a),\mathbb{Q}_k(b)[m])\cong 
\op{Hom}_{\mathscr{T}(k)}(\op{M}_k(k),\mathbb{Q}_k(b-a)[m])=0
$$
for $m< 0$ (or in the stronger version for $m\leq 0$ and $b\neq a$). 
By definition, every object $M$  of
$\DMT(\mathscr{T},k)$ can be constructed from $\mathbb{Q}_k(n)$,
$n\in\mathbb{Z}$ using triangles. The corresponding long exact
sequences then yield the claim. 
\end{proof}

Most of the time, we will apply the above result to flag varieties
$G/B$, Bruhat cells $BxB/B\cong\mathbb{A}^n$ in flag varieties or
Bott--Samelson resolutions of Schubert varieties. These varieties have
motivic cell structures, hence their motives are mixed Tate.

\begin{remark}
 The Beilinson--Soul{\'e} vanishing conjecture holds for finite
  fields by the K-theory computations of Quillen, cf. \cite{quillen}.
The Beilinson--Soul{\'e} vanishing holds for global fields - for number
fields by the K-theory computations of Borel, and for function fields
by the group homology computations of Harder, cf. \cite{harder}.
\end{remark}

\begin{remark}
Finally, we want to remark that the Beilinson--Soul{\'e} condition for
$\mathscr{T}(\mathbb{C})$ given by semisimplified Hodge realization is a
triviality. This holds more generally whenever $\mathscr{T}$ satisfies
the grading condition - the motivic t-structure is then the natural 
t-structure on the category of $\mathbb{Z}$-graded
$\mathbb{C}$-vector spaces. This is one of the reasons why the motives
with coefficients in semisimplified Hodge realization are useful: we
get results over the base field $\mathbb{C}$ where the
Beilinson--Soul{\'e} conjectures for \'etale or Beilinson motives are
not known.  
\end{remark}

\subsection{Summary of structures}

Let $k$ be a field satisfying the Beilinson--Soul{\'e} vanishing
conjectures. The category $\DMT(k)$ is a tensor triangulated category,
equipped with a weight structure and a t-structure. 

The results of Wildeshaus \cite[Th{\'e}or{\`e}me 1.1, Corollaire
1.4]{wildeshaus:cras1} imply that  there is an exact functor 
$$\op{real}:\op{Der}^{\op{b}}(\op{MT}(k))\to \DMT(k)
$$
which is an equivalence of categories and induces the identity on
$\op{MT}(k)$. Note that the result in loc.cit. is stated for number
fields, but all that is required for the proof is the
Beilinson--Soul{\'e} vanishing. 

The above-mentioned identification $W_{[n,n]}\DMT(k)\approx
\op{Der}^{\op{b}}(\mathbb{Q}\op{-modf})$ restricts to an equivalence
$W_{[n,n]}\op{MT}(k)\approx \mathbb{Q}\op{-modf}$. By \cite[Theorem
4.2]{levine:nato}, this equivalence provides an exact faithful tensor
functor 
$$
\bigoplus_{i\in\mathbb{Z}}\op{gr}^W_i: \op{MT}(k)\to
\mathbb{Q}\op{-modf}^{\mathbb{Z}}
$$
from the  category of mixed Tate motives to the category of
finite-dimensional graded $\mathbb{Q}$-vector spaces. In
the special cases where $k$ is a finite field or a global function
field, the vanishing of rational K-theory allows to identify $\op{MT}(k)$
with the category of finite-dimensional graded $\mathbb{Q}$-vector
spaces. The result is an identification of $\DMT(k)$ with the
bounded derived category of finite-dimensional graded
$\mathbb{Q}$-vector spaces.   

In the case of a field $k$, there are now two ways of
defining weights for mixed Tate motives. The comparison between these
two is given by \cite[Theorem 3.8]{wildeshaus:at}: a mixed Tate
motive $M$ is in $\DMT(k)_{w=0}$ if and only if
$\op{H}^i(\op{gr}^W_j M)=0$ for $i\neq 2j$. The weight structure $w$
assigns weight $q-2p$ to the motives $\mathbb{Q}(p)[q]$, the weight
t-structure $W$ assigns weight $p$. 

In the case of a finite field (again using vanishing of rational
K-theory), the motives of $w$-weight $0$ form a tilting
collection. This provides another equivalence  of triangulated categories
$\op{Der}^{\op{b}}(\DMT(k)_{w=0})\cong \DMT(k)$. The result is
an easy version of Koszul duality that ``interchanges the weight and
t-structure.'' It is the unique triangulated self-equivalence that
maps $\mathbb{Q}(n)$ to $\mathbb{Q}(-n)[-2n]$: the first object has
cohomological degree $0$ and weight $-2n$, the latter has
cohomological degree $-2n$ and weight $0$. The results of our paper
can be interpreted as saying that the Koszul duality of \cite{BGSo}
for stratified mixed Tate motives over partial flag varieties is
essentially obtained by perverse glueing from this toy example. 

It is interesting to note that in the case of a number fields, the
hearts of the weight and t-structure are not equivalent. The heart
of the weight structure is semi-simple, while the heart of the
t-structure has a lot of interesting arithmetic extensions of Tate
motives. A functor as above still exists and embeds the heart of the
t-structure into the heart of the weight structure, splitting the
extensions. It is hence not exactly clear if the above Koszul duality
functor can have a ``geometric construction''. We thank J\"org
Wildeshaus for discussions on this point.

\section{Stratified mixed Tate motives}
\label{sec:dmt}

In the following section, we consider categories of motives over
stratified varieties. We want to study motives which are constant
mixed Tate along the strata. For this, we need a condition which, in
analogy with the case of sheaves on topological spaces, we call
{\bf Whitney--Tate}. This condition is in particular satisfied for
partial flag varieties with the stratification by Schubert cells. A
further discussion of Whitney--Tate stratifications is deferred to Appendix~\ref{sec:WTS}.

\begin{convention}
\label{conditions}
From this moment on, we will consider motivic triangulated categories
$\mathscr{T}$ over $\mathscr{S}=\op{Sch}/k$, i.e., we will only work
over schemes separated and of finite type over some field. All the
constructions  will take place in the motivic triangulated
category $\mathscr{T}$, which will sometimes be suppressed from the
notation. In particular, whenever we speak of \emph{motives}, we are
referring to objects in some category $\mathscr{T}(X)$ where hopefully
the exact nature of $\mathscr{T}$ will be clear from context.

For the representation-theoretic applications, we will usually
consider a more restricted setting in which the following
two additional conditions are satisfied:
\begin{description}
\item[(weight condition)] A motivic triangulated category $\mathscr{T}$ over   $\op{Sch}/k$ is said to satisfy the \emph{weight condition} if for each scheme $X$ there is a weight structure on $\mathscr{T}(X)$, such that this collection of weight structures satisfies the conclusion of H{\'e}bert's theorem  \ref{thm:hebert}. 
\item[(grading condition)] A motivic triangulated category $\mathscr{T}$ (with coefficients in a field $\mathbb{K}$ of characteristic $0$) over $\op{Sch}/k$  is said to satisfy the \emph{grading condition} if $\DMT(\mathscr{T},k)$ is  equivalent (as tensor-triangulated category) to the bounded derived category of finite-dimensional $\mathbb{Z}$-graded $\mathbb{K}$-vector spaces.
\end{description}

From the discussion in \prettyref{sec:cdbmot} and \prettyref{sec:mtm},
these two conditions are satisfied for 
rational motives over finite fields and for
$\mathcal{E}_{\op{GrH}}$-motives over $\mathbb{C}$. These are the
situations of interest for our representation-theoretic applications. 
\end{convention}

\begin{definition}\label{defin:cellvar} 
 By a {\bf stratification} of a variety we mean a finite partition 
$$
X=\bigsqcup_{s\in\mathcal{S}}X_s
$$
of $X$ into locally closed smooth subvarieties, called the {\bf strata} of our stratification, such that the closure of each stratum  is again a union of strata. If all strata are isomorphic to affine spaces $\mathbb{A}^{n_s}$ of some dimension depending on $s\in\mathcal{S}$, we speak of a {\bf stratification by affine spaces} or of an {\bf affinely stratified variety}. 
\end{definition}

  \begin{Bemerkungl}
\label{DMTk}
    Given a stratified variety $(X,\mathcal{S})$ we
    consider the full
    triangulated
    subcategories $$\DMT^*_{\mathcal{S}}(X),\DMT^!_{\mathcal{S}}(X)\subset
    \mathscr{T}(X)$$  of all  motives $M$ 
    such that for each inclusion $j_s:X_s\hookrightarrow X$ of a stratum
    $j_s^\ast M$ respectively $j_s^! M$ belongs to $\DMT(X_s)$.
  \end{Bemerkungl}

\begin{lemma}\label{lem:genZ} 
  Given a  stratified variety $(X,\mathcal{S})$
 the category $\DMT^*_{\mathcal{S}}(X)$ 
is generated  as a triangulated 
category by the objects $j_{s!}M$ for
$s\in \mathcal S$ and $M\in \DMT(X_s)$. Similarly
$\DMT^!_{\mathcal{S}}(X)$ is generated by the objects $j_{s*}M$.
\end{lemma}
\begin{proof}
We prove the first statement, the second is similar. 
We argue by induction on the number of strata, the case of no stratum
 being obvious. 
Let $j_s:X_s\hra X$ be the inclusion of an open stratum and
$i:Z\hra X$ the inclusion of its complement. 
For $M\in \DMT_{\mathcal{S}}^\ast(X)$ consider the
``Gysin'' or ``localization'' triangle 
$$j_{s!}j_s^*M\ra M\ra i_!i^*M\ra j_{s!}j_s^*M[1].
$$  
Obviously, $j_s^\ast M\in\DMT(X_s)$, and so the first term is of
the required form. On the other hand, $i^\ast M\in
\DMT_{\mathcal{S}}^{\ast}(Z)$ and the induction hypothesis implies that
$i^\ast M$ is built from motives $k_{t!}N$ with $k_t:Z_t\hra Z$ a
stratum of $Z$ and $N\in\DMT(Z_t)$. Hence $i_!i^\ast M$ is of the
required form, and  the claim is proved. 
\end{proof}

\begin{definition}
\label{defin:WT}
  A stratified variety $(X,\mathcal S)$ is called {\bf Whitney--Tate}
  if and only if for all $s,t\in\mathcal S$ and $M\in \DMT(X_s)$ we
  have $j^{*}_tj_{s*}M\in \DMT(X_t)$. 
\end{definition}

\begin{remark}
By Verdier duality, this condition   is equivalent to asking 
 $j^{!}_tj_{s!}M\in \DMT(X_t)$. Using  \prettyref{lem:genZ} we deduce
 in this case the equality
 $\DMT^!_{\mathcal{S}}(X)=\DMT^*_{\mathcal{S}}(X)$. 
\end{remark}

\begin{definition}
\label{defin:dmts}
Given a Whitney--Tate stratified variety, the category
$$
\DMT^!_{\mathcal{S}}(X)=\DMT^\ast_{\mathcal{S}}(X)=
\DMT_{\mathcal{S}}(X)\subset \mathscr{T}(X)$$  
is called the {\bf category of stratified mixed Tate motives}.
\end{definition}

  \begin{remark}
    Similar categories have appeared before, in the setting of $\ell$-adic
    sheaves in \cite[Section 4.4]{BGSo}, and in the setting of Tate motives in
    \cite[Section 4, Theorem 4.4]{wildeshaus:intermediate}. In
    particular, \cite[Theorem 4.4]{wildeshaus:intermediate} states
    that an affinely stratified variety is Whitney--Tate if the orbit
    closures are regular.  
  \end{remark}

\begin{remark}
Let us recall some facts concerning the Bruhat decomposition of a split reductive group $G$. If $T\subset B$ is a maximal torus in a Borel subgroup of $G$, multiplication gives an isomorphism of varieties $T\times B_{\op{u}}\sira B$, where $B_{\op{u}}$ is the unipotent radical of $B$. If  $U_\alpha\subset G$ are the root subgroups for $T$, which are all isomorphic to the additive group, and $R^+$  is the system of positive roots of $T$ in $B$, multiplication gives an isomorphism of varieties $\prod_{\alpha\in R^+}U_\alpha\sira B_{\op{u}}$ and an open embedding $\prod_{\alpha\in R^+}U_{-\alpha}\times B\hra G$ for the products taken in an arbitrary but fixed order. For $W\supset W_P$ the Weyl groups of $G$ and a parabolic subgroup $P$, respectively, we have the Bruhat decompositions 
$$
G=\bigsqcup_{x\in W} BxB, \quad P=\bigsqcup_{x\in W_P} BxB \quad\textrm{ and }\quad G=\bigsqcup_{\bar  y\in W/W_P} B\bar yP.
$$ 
The double cosets can be described quite explicitly by isomorphisms
$$
\prod_{\alpha\in R^+\backslash yR^+}U_\alpha\times \{\dot y\}\times P
\sira B\bar yP
$$ 
given again by multiplication with $y\in W$ the shortest representative of $\bar y$ and $\dot y$ representing $y$.  In particular, $G\ra G/P$ and also all $G/Q\ra G/P$ for inclusions $Q\subset P$ of parabolic subgroups are  trivial fibre bundles over any cell $B\bar yP/P$. Moreover, the preimage in $G/Q$ of the cell  $B\bar yP/P$ is the union of cells $B\bar xQ/Q$ for  $\bar  x\in W/W_Q$ satisfying $\bar xW_P= \bar y$. The maps induced between these cells are trivial fibre bundles with affine spaces as fibres. For more details, one might consult \cite{BorAG}. 
\end{remark}

  \begin{proposition}
\label{prop:erpt2}
Let $G$ be a connected split reductive algebraic group over the field $k$. 
Let $T \subset  B\subset P\subset G$ be a choice of split maximal
torus $T$, a Borel subgroup $B$ and a parabolic subgroup $P$.  Then
the stratification  of $G/P$ by $B$-orbits is Whitney--Tate.  
\end{proposition}

\begin{proof}
Let us first concentrate on the case $P=B$.

Recall that Bruhat cells are parametrized by the elements of the Weyl group $W$ of $G$. Given an element $t$ in the Weyl group, let $X_t=BtB/B$ be the corresponding Bruhat cell. For every simple reflection $s$ in $W$, corresponding to a simple positive root $\alpha$ w.r.t. $B$, denote by $P_s$ the parabolic corresponding to the set $\{\alpha\}$, i.e., the subgroup generated by the Borel and the root subgroup for $-\alpha$. The inclusion $B\subset P_s$ induces a projection $\pi_s:G/B\ra G/P_s$ from $G/B$ onto the  partial flag variety for the parabolic $P_s$. The projection $\pi_s$ is Zariski-locally trivial with fiber $P_s/B\cong\mathbb{P}^1$. Now consider the pullback square 
\begin{center}
  \begin{minipage}[c]{10cm}
    \xymatrix{
      X_t\sqcup X_{ts}\ar@{=}[r]&Y \ar[r]^{u} \ar[d]_{p} &
      G/B \ar[d]^{\pi_s}\\ 
     & D\ar[r]_{v} & G/P_s 
    }
  \end{minipage}
\end{center}
where $D$ is a $B$-orbit in $G/P_s$. By what was said above, $Y\cong\mathbb{P}^1\times D$ and it decomposes into two $B$-orbits as shown. Without loss of generality, the $B$-orbits correspond to the Weyl group elements $t$ and $ts$ with $t<ts$ in the Bruhat order, and with this choice we have  $X_{ts}\cong \mathbb{A}^1\times D$ and $X_t\cong \{\infty\}\times D$. We denote the open immersion $j:X_{ts}\hra Y$ and the closed immersion $i: X_t\hra Y$. In the following, we denote $j_t:X_t\to G/B$ the inclusion of cells in $G/B$. 
The projection induces an isomorphism $p:X_t\sira D$, thus we get $p^*p_*\underline X_t\cong i^\ast \underline Y$ and applying $u_!$ and base change we get  
$\pi_s^*\pi_{s*}j_{t!}\underline X_t\cong u_!\underline Y$.
On the other hand we have a triangle
$j_!j^*\underline Y\ra \underline Y\ra i_!i^*\underline Y\ra
j_!j^\ast\underline Y[1]$
 and with $u_!$ a triangle
$$j_{ts!}\underline X_{ts}\ra \pi_s^*\pi_{s*} 
j_{t!}\underline X_t\ra j_{t!}\underline X_t\ra
j_{ts!}\underline X_{ts}[1]$$
on $G/B$. This shows that any triangulated subcategory of
$\mathscr{T}(G/B)$ stable under all $\pi_s^*\pi_{s*}$ for all simple
reflections $s$ and
containing  the skyscraper $j_{e!}\underline X_e$ at the one-point cell $X_e$ 
has to contain   $j_{r!}\underline X_r$ for all Bruhat cells $BrB/B$. 
Now it is sufficient to see that our triangulated subcategory 
$\DMT_{(B)}^!(G/B)$ from \ref{DMTk} has all these properties,
since then all $j_{r!}\underline X_r$ belong to it and our stratification is indeed Whitney--Tate.
Given $M\in \DMT_{(B)}^!(G/B)$ we thus need to show $ \pi_s^*\pi_{s*}M\in \DMT_{(B)}^!(G/B)$, i.e., the $!$-restriction of $u^!\pi_s^\ast \pi_{s\ast}M$ to $X_t$ resp. $X_{ts}$ is constant mixed Tate. Since $\pi_s$ is smooth, we can exchange $u^!\pi_s^\ast\cong p^\ast v^!$, and base change allows to exchange $v^!\pi_{s\ast}\cong p_\ast u^!$. Hence, it will be sufficient to show $p^*p_{*}u^!M\in \DMT_{(B)}^!(Y)$.
Now sure enough we have $N\pdef u^!M\in  \DMT_{(B)}^!(Y)$.
We consider the push-forward of the localization triangle:
$$
p_{*}i_*i^!N\ra p_{*}N\ra p_{*}j_*j^!N\ra p_{*}i_*i^!N[1].
$$
Here  $p\circ i$ is an isomorphism and $p\circ j$ is the projection $\mathbb{A}^1\times D\to D$, as remarked above. By the homotopy property we deduce $p_{*}N\in \DMT(D)$, and then $p^*p_{*}N\in \DMT^!_{(B)}(Y)$ follows easily.

The case of a general parabolic $P$ can be deduced from the Borel case above as follows: the inclusion $B\subset P$ induces a projection $\pi:G/B\to G/P$ which is smooth and projective; actually, it is a Zariski-locally trivial bundle with fiber $P/B$ and for any cell $X_t=BtP/P\subset G/P$, we have $\pi^{-1}(X_t)\cong X_t\times P/B$. To determine the restriction $j_t^\ast j_{s\ast}M$ from a cell $X_t$ to a cell $X_s$ of $G/P$, we can use the following diagram 
\begin{displaymath}
  \xymatrix{
    X_t\times P/B  \ar@{^{(}->}[r]\ar[d] & G/B \ar[d] & X_s\times P/B \ar[d]  \ar@{_{(}->}[l] & X_s\ar@{_{(}->}[l] \ar@{=}[ld]\\
     X_t \ar@{^{(}->}[r]_j & G/P & X_s \ar@{_{(}->}[l]
  }
\end{displaymath}
Moreover, as in the proof for the case $B$ above, there is a Bruhat-cell of $X_s\times P/B$ which is isomorphic to $X_s$. To compute $j_t^\ast j_{s\ast}M$, it suffices, by base change,  to pull back $M$ to $X_t\times P/B$, extend to $G/B$ and restrict to $X_s\subset X_s\times P/B$. Since $G/B$ is Whitney--Tate, the result is a constant mixed Tate motive on $X_s$, and we are done.
\end{proof}

\begin{remark}
Further conditions for a stratification to be Whitney--Tate can be
found in Appendix~\ref{sec:WTS}. These conditions allow another proof of the above \prettyref{prop:erpt2}, using that fibres of Bott--Samelson resolutions  of Schubert cells have mixed Tate motives, cf. \prettyref{prop:erpt1}.  
\end{remark}

\begin{example} 
\label{ex:paraboloid}
Let $G$  be a connected split reductive algebraic group over the field $k$. 
Let $T \subset  B\subset P\subset G$ be a choice of split maximal
torus $T$, a Borel subgroup $B$ and a parabolic subgroup $P$.  
It is well-known that the partial flag varieties $G/P$ are 
affinely stratified by the $B$-orbits alias Schubert cells.
By a {\bf paraboloid $B$-variety}, we mean a $B$-variety $Y$ which 
is isomorphic to a locally closed $B$-stable subset of a partial flag
variety $G/P$.  Plainly, these are affinely stratified
by $B$-orbits  as well.
In this case we denote the stratification by $(B)$ and call 
the objects of $\DMT_{(B)}(Y)$ {\bf Bruhat--Tate sheaves}. By the
arguments in \prettyref{sec:WTS}, the Bruhat stratifications of
paraboloid $B$-varieties are also Whitney--Tate. 
\end{example}

\begin{Bemerkungl}
  Other examples of affinely stratified varieties can be found among smooth
  projective spherical varieties, Hessenberg varieties and symmetric
  spaces. In all these cases, locally closed cells arise from the
  Bia{\l}ynicki-Birula decomposition associated to suitably chosen
  $\mathbb{G}_{\op{m}}$-actions and in most cases of interest, these
  also give rise to stratifications.
\end{Bemerkungl}

\section{Weight structure for stratified mixed Tate motives}
\label{sec:weights}

In the following section, we discuss the existence and properties of a weight structure on the category $\DMT_{\mathcal{S}}(X)$ of stratified mixed Tate motives for a Whitney--Tate stratified variety $(X,\mathcal{S})$. 
We fix a motivic triangulated category $\mathscr{T}$, which is
required to satisfy the weight condition - so that we can talk about weight
structures. For the later results on combinatorial models for the
heart, we will additionally need the grading condition, but this is not
required for the definition of the weight structure. 

\begin{proposition}
\label{prop:hebbswz} 
Let $(X,\mathcal{S})$ be an affinely Whitney--Tate stratified variety. 
Then  on the category $\DMT_{\mathcal{S}}(X)$ of stratified mixed Tate
motives, cf. \prettyref{defin:dmts}, we obtain a weight structure $w$
by setting
$$
\begin{array}{l}
\DMT_{\mathcal{S}}(X)_{w\leq 0} \pdef \left\{
M\mid j_s^\ast M \in \DMT(X_s)_{w\leq 0}\text{ for all strata
}s\in\mathcal S\right\}

\\[2mm]
\DMT_{\mathcal{S}}(X)_{w\geq 0} \pdef \left\{
M\mid j_s^! M \in \DMT(X_s)_{w\geq 0}\text{ for all strata
}s\in\mathcal S\right\}
\end{array}
$$
This weight structure coincides with the restriction of H{\'e}bert's
weight structure on $\mathscr{T}(X)$ to $\DMT_{\mathcal{S}}(X)$.
\end{proposition}

\begin{proof}
  To prove the existence of such a weight structure
we proceed by induction on the number of
  strata. If there is no stratum, the claim is correct. Otherwise,
decompose $X$ as the disjoint union of an open stratum $j:X_s\hra X$ 
and its closed
complement $i:Z\hra X$.  Using Bondarko's result \cite[Proposition 1.7 (13),
(15)]{bondarko:imrn} on glueing weight structures, we obtain a weight
structure on $\DMT_{\mathcal{S}}(X)$  by setting
$$
\begin{array}{l}
\DMT_{\mathcal{S}}(X)_{w\leq 0}\pdef \left\{M\mid i^\ast M\in
  \DMT(Z)_{w\leq 0}, \;
  j^\ast M\in \DMT(X_s)_{w\leq 0}\right\}

\\[2mm]
\DMT_{\mathcal{S}}(X)_{w\geq 0}\pdef\left\{M\mid i^!M\in
  \DMT(Z)_{w\geq 0}, \;
  j^! M\in \DMT(X_s)_{w\geq 0}\right\}  
\end{array}$$
Now recall that for any separated finite type morphism $f$, the functors
$f^\ast$ and $f^!$ are left and right weight-exact, respectively, for
H{\'e}bert's weight structure. 
This implies that objects of weight $\leq 0$ for H{\'e}bert's
 weight structure are
also of weight $\leq 0$ for our weight structure, 
and similarly for $\geq 0$. 
For the reverse inclusions, we use the same induction. 
Assume the result is established for $Z$. 
 By
\cite[Proposition 1.7 (13)]{bondarko:imrn}, the $(w\leq 0)$-part of the
glued weight structure on $\DMT_{\mathcal{S}}(X)$ is generated
by $j_!\DMT(X_s)_{w\leq 0}$ and $i_\ast
\DMT_{\mathcal{S}}(Z)_{w\leq 0}$.
 This implies all its objects also belong to
the $(w\leq 0)$-part of H{\'e}bert's weight structure. 
 A dual argument takes care of the $(w\geq
0)$-part of the weight structures.  
Finally, it also follows directly from the above arguments that the
weight structure constructed this way has the 
description claimed in the statement of the proposition.
\end{proof}

\begin{remark}
  This  generalizes  \cite[Corollary
  4.12]{wildeshaus:intermediate} to some cases where closures of
  strata are not necessarily regular.
\end{remark}

\section{Pointwise purity, Bott--Samelson motives and the heart}
\label{sec:ptwise1}

In the next section, we investigate the heart of the weight structure
defined in \prettyref{sec:weights}, in the special case of flag
varieties. We show that motives of Bott--Samelson resolutions of
Schubert cells satisfy an additional property called {\bf pointwise
  purity} and deduce that the heart of the weight structure is
generated by motives of Bott--Samelson resolutions.

\begin{definition}
Let $(X,\mathcal{S})$ be an affinely Whitney--Tate 
stratified variety. A stratified Tate motive
$M\in\DMT_{\mathcal{S}}(X)_{w=0}$ is called {\bf pointwise $\ast$-pure}
 if for each inclusion $i_s:X_s\to X$ of a stratum, we have
$i_s^\ast M\in \DMT(X_s)_{w=0}$. Similarly, we define the concept
{\bf pointwise $!$-pure}. If both conditions are satisfied, the motive
is called  {\bf pointwise pure}.\label{poip}
\end{definition}

\begin{proposition}\label{prop:ptDI}   
Let $(X,\mathcal{S})$ be an affinely Whitney--Tate 
stratified variety, and 
denote by $\op{fin}:X\to\op{pt}$ the structure morphism. For 
any pointwise $\ast$-pure stratified Tate  
motive $M\in \DMT_{\mathcal{S}}(X)$, the
object $\op{fin}_! M$ is pure
Tate of weight $0$, in formulas 
$\op{fin}_!M\in \DMT(\op{pt})_{w=0}$.  
\end{proposition}

\begin{proof}
The statement is proved by induction
on the number of strata. If there is no stratum, the claim is evident. 
For the inductive step, consider the embedding $j:X_s\to 
X$ of an open stratum and let $i:Z\hra X$ be the embedding of its complement.
We have the localization sequence 
\begin{equation*}
  j_!j^\ast  M\cong j_!j^! M \rightarrow M
  \rightarrow i_! i^\ast  M 
  \cong i_\ast i^\ast  M\to j_!j^*  M[1].
\end{equation*}
After proper pushforward, this sequence becomes 
$$  
\op{fin}_!j^\ast M \rightarrow \op{fin}_!M
\rightarrow \op{fin}_! i^\ast M \to \op{fin}_!j^*  M[1].
$$
By induction we may assume $\op{fin}_! i^\ast
M\in\DMT(\op{pt})_{w=0}$. 
On the other hand, the homotopy property
 implies that since $X_s\cong \mathbb A^n$, 
the pushforward $\op{fin}_\ast:\DMT(X_s)\to
 \DMT(\op{pt})$ is in fact an equivalence which is compatible with the
 weight structures and duality. Therefore, $(\op{fin}_s)_!j^\ast M\in
 \DMT(\op{pt})_{w=0}$.   
 By \cite[Proposition 1.7(2)]{bondarko:imrn},
hearts of weight structures are extension-stable, so
$\op{fin}_! M\in \DMT(\op{pt})_{w=0}$. 
\end{proof}

\begin{corollary}\label{cor:bswz}  
Let $(X,\mathcal{S})$ be an affinely Whitney--Tate stratified variety. Given $M,N\in \DMT_{\mathcal{S}}(X)$ with $M$  pointwise $*$-pure and $N$ pointwise $!$-pure we have $\mathscr{T}_{X}(M,N[a])=0$ for any $a>0$.
\end{corollary}


\begin{proof}
We first note, that point (iv) of \cite[Th{\'e}or{\`e}me 3.7]{hebert} has a $\op{Hom}$-analogue, and implies that on the category $\DMT(k)$ of mixed Tate motives, the functor $\op{Hom}$ is in fact weight-exact. Using this, we see that for any stratum, $j_s^!\op{Hom}(M,N)\cong \op{Hom}(j_s^* M,j_s^! N)$ is pure of weight zero by the assumption, thus  $\op{Hom}(M,N)$ is pointwise  $!$-pure. By \prettyref{prop:ptDI} we deduce that its direct image $\op{fin}_*\op{Hom}( M,N)$ is pure of weight zero. This in turn means  $$\mathscr{T}_{\op{pt}}(\underline{\op{pt}},\op{fin}_*\op{Hom}( M,N)[a])=0$$ for $a>0$ by the definition of a weight structure. But this is just another way to write the space we claim to vanish.
\end{proof}

\begin{definition}
Let $k$ be a field and let $G\supset B\supset T$  be a split reductive
group with a choice of maximal torus $T$ and Borel subgroup $B$.
We define the collection of full
subcategories 
\begin{equation*}
  \DMT^{\op{bs}}_{(B)} (G/Q) \subset \DMT_{(B)}(G/Q)
\end{equation*}
for all standard parabolic subgroups $B\subset Q \subset G$ to be
the smallest collection with the following properties:
\begin{enumerate}
\item the collection contains the skyscraper at the one-point-cell of
  $G/B$, i.e., for $j_e:\op{pt}\hra G/B$ the embedding of the 
  $B$-orbit $B/B$, we have  $(j_e)_*\underline{\op{pt}} \in
  \DMT^{\op{bs}}_{(B)} (G/B)$,
\item the collection is stable under $M\mapsto M(n)[2n]$ and direct
  summands,
\item the collection is extension-stable in the sense 
that for a distinguished triangle $A\to B\to C\to A[1]$ with $A$ and
$C$ in the subcategory, $B$ is also in the subcategory, and 
\item if $\pi : G/P \rightarrow G/Q$ is a projection for standard
  parabolic subgroups $P \subset Q$, then we have
  \begin{displaymath}
    \begin{array}{ccc}
      M \in \DMT^{\op{bs}}_{(B)} (G/P) &\Rightarrow &\pi_! M,\pi_\ast M
      \in \DMT^{\op{bs}}_{(B)} (G/Q)\\[2mm] 
      M\in \DMT^{\op{bs}}_{(B)} (G/Q) &\Rightarrow & \pi^! M, \pi^\ast
      M \in \DMT^{\op{bs}}_{(B)} (G/P) 
    \end{array}
  \end{displaymath}
\end{enumerate}
\end{definition}

\begin{Bemerkungl}
  It is not difficult to see that the direct images of the constant
motives on Bott--Samelson resolutions all belong to 
$\DMT^{\op{bs}}_{(B)} (G/P)$. We will call the objects of this category
the {\bf Bott--Samelson motives}. 
\end{Bemerkungl}

\begin{lemma}
\label{lem:bswz}
Bott--Samelson motives are pointwise pure.
\end{lemma}

\begin{proof}
Pointwise purity is obviously satisfied for $(j_e)_*\underline{\op{pt}} \in  \DMT^{\op{bs}}_{(B)} (G/B)$  
and is stable under $M\mapsto M(n)[2n]$ and direct summands. It is also extension-stable, because the heart of the weight structure on  $\DMT(\op{pt})$ 
is extension-stable. It then suffices to show that pointwise purity is stable under push-forwards and pullbacks along projections $\pi:G/P\to G/Q$ where $P\subset Q$ are any two parabolic subgroups of $G$. Recall that $\pi$ is a Zariski-locally trivial fiber bundle with fiber $Q/P$. 

For pullbacks this is more or less evident: let $M\in \DMT^{\op{bs}}_{(B)}(G/Q)$ and assume that for each stratum $j_s:X_s\to G/Q$, we have $j_s^\ast M, j_s^!M\in \DMT(X_s)_{w=0}$. We want to show that for each stratum $j_t:X_t\to G/P$, we have $j_t^\ast \pi^\ast M\in \DMT(X_t)_{w=0}$. The projection is $B$-equivariant, and the fiber $X_s\times Q/P$ of $\pi$ over $X_s$ is a union of $B$-orbits in $G/P$. From the evident commutative diagram and the fact that $\DMT(X_s)\cong \DMT(k)$ because $X_s$ is an affine space, we find that $j_t^\ast \pi^\ast M$ is the restriction of a motive from $\DMT(X_s\times Q/P)_{w=0}$. Evidently, $j_t^\ast \pi^\ast M$ is pure of weight $0$ for every stratum $j_t:X_t\to G/P$.  By relative purity applied to the smooth projection $\pi:G/P\to G/Q$, the same statement also holds for $\pi^! M$. 

We next consider the direct image functors. The inclusion of a $B$-orbit $j: D \hookrightarrow G/Q$ can be embedded into a commutative diagram
\begin{displaymath}
  \xymatrix{
    Q/P \ar[d] &Y \ar[l] \ar@{^{(}->}[r]^{j'}\ar[d]_{\pi'} & G/P \ar[d]^\pi\\
    \op{pt} & D\ar[l] \ar@{^{(}->}[r]_j & G/Q
  }
\end{displaymath}
in which both squares are pullback squares. As noted above, $Q/P$ is the fiber of $\pi$ over any point of $D$, and $Y\cong \pi^{-1}(D)$. By the $B$-equivariance of $\pi$, the $B$-orbits in $Y$ are precisely the inverse images of the $B$-orbits in $Q/P$. By base change, we have $j^\ast\pi_\ast  M\cong (\pi')_\ast (j')^\ast M$, and, because $\pi$ is smooth and projective, similar statements for the other pullbacks and push-forwards. Since $D$ is an affine space, we can identify $\DMT(\op{pt})\cong\DMT(D)$ and $\DMT(Q/P)\cong\DMT(Y)$, hence we are reduced to the case $Q/P$ projecting to a point. But by \prettyref{prop:ptDI} we know that $\op{fin}_!M\in \DMT(\op{pt})_{w=0}$ for any pointwise $\ast$-pure motive $M\in \DMT_{(B)}(Q/P)$. 
This shows that pointwise purity is stable under direct image functors and finishes the proof. 
\end{proof}

\begin{corollary}
\label{cor:bsgen}
There is an equality of full subcategories
$$
\DMT^{\op{bs}}_{(B)}(G/P)= \DMT_{(B)}(G/P)_{w=0}.
$$ 
\end{corollary}

\begin{proof}
From \prettyref{cor:bswz}, we find that 
$\DMT^{\op{bs}}\pdef \DMT^{\op{bs}}_{(B)}(G/P)$ is
negative in the sense of \cite{bondarko:imrn}. An induction on the
dimension of the partial flag varieties shows that the smallest
triangulated idempotent complete subcategory of $\DMT_{(B)}(G/P)$ which
contains $\DMT^{\op{bs}}$ is $\DMT_{(B)}(G/P)$ itself. 
From \cite[Proposition 1.7(6)]{bondarko:imrn}, there is a unique
weight structure on $\DMT_{(B)}(G/P)$ such that
$\DMT^{\op{bs}}$ is pure of weight $0$. By
\prettyref{prop:hebbswz}, the weight structure defined by
$\DMT^{\op{bs}}$ has to coincide with the weight structure of
H{\'e}bert. The heart of this weight structure is then
$\DMT^{\op{bs}}$, by \cite[Proposition 1.7(6)]{bondarko:imrn}. 
 \end{proof}


\section{Pointwise purity via equivariance} 
\label{sec:ptwise2}

In this section, we discuss another way of establishing the condition
of pointwise purity that was so crucial in identifying the objects of
the heart in \prettyref{sec:ptwise1}. We adapt an argument of Springer
\cite{Spp} to the motivic setting, showing that suitably equivariant
motives on locally $\mathbb{A}^1$-contractible $G$-varieties are
pointwise pure. 

\begin{definition}
\label{defin:wequiv}
 Given a variety $X$ with an action of an algebraic group $G$, a
 motive $M \in \mathscr{T} (X) $ is called 
{\bf weakly $G$-equivariant} if and only if
there exists an isomorphism
\begin{equation*}
\op{act}^\ast M \cong \op{pr}^\ast M
\end{equation*}
of motives in $\mathscr{T} (G \times X)$. 
Here $\op{act} ,\op{pr} : G \times X \rightarrow X$ denote
the action and projection map, respectively.
\end{definition}

\begin{remark}
  We want to stress that the isomorphism in \prettyref{defin:wequiv} is
  {\bf not} part of the data, nor do we require  any 
compatibilities for it. Therefore, the  condition of weak
$G$-equivariance is indeed quite weak. 
By proper and smooth base change, we see easily 
that weak equivariance is preserved
under $f^\ast, f_\ast, f^!, f_!$ for any $G$-equivariant morphism
$f$. In particular, the Bott--Samelson motives of
\prettyref{sec:ptwise1} are weakly $G$-equivariant.
\end{remark}

The following is a
straightforward translation of arguments of Springer \cite[Proposition 1
and Corollaries]{Spp}, repeating \cite[1.3]{So-N}.
This leads to an alternative proof that Bott--Samelson motives are
pointwise pure. 

\begin{proposition}
\label{prop:wep}  
 Let $X$ be a variety, let $Z\subset X$ be a closed subvariety, and
 assume that there exists an action $\mathbb{G}_{\op{m}}\times X\to X$ which
 contracts $X$ onto $Z$. Let $M \in \mathscr{T}(X)$ be weakly
$\mathbb G_{\op{m}}$-equivariant. Let $a : Z \hookrightarrow X$ denote 
the inclusion and 
$ p : X \rightarrow Z$ the morphism mapping each point to its limit.
Then in $\mathscr{T}(Z)$ there
exists an isomorphism
\begin{equation*}
 p_\ast M \cong a^\ast M.
\end{equation*}
\end{proposition}

\begin{proof}
 We will prove the stronger claim that the adjunction map $M
 \rightarrow a_\ast  a^\ast M$ becomes an isomorphism after applying
 $p_\ast$. To 
 prove this, let $b : U \hookrightarrow X$ be the open embedding of
 the complement of $Z$. By the localization sequence
 $$
b_!b^! M \rightarrow  M \rightarrow a_\ast a^\ast
M \to b_!b^!M[1]
$$ it will be sufficient to show
$p_\ast b_! b^! M = 0$.

In fact, we will show $p_\ast N = 0$  for any weakly equivariant $N
\in \mathscr{T} (X)$ with $a^\ast N = 0$. The strategy is to construct an
automorphism of $p_\ast N$ that factors through zero. 
For a $\mathbb G_{\op{m}}$-action to contract to a subvariety $Z$ means
 that the action $ \mathbb G_{\op{m}} \times X \rightarrow X$
can be extended to a morphism 
$\op{act}: \mathbb A^1 \times X \rightarrow X$ such that
we have $\op{act}\circ \kappa=a\circ p$ where
$\kappa:X\to\mathbb{A}^1\times X:x\mapsto (0,x)$ is the $0$-section,
and   $p:X\ra Z$ is the morphism mapping each point to its limit.
Consider now the morphism
$$\begin{array}{cccl}
 \tau :& \mathbb A^1 \times X & \rightarrow & \mathbb A^1 \times X\\
&(t,x) &\mapsto & (t, \op{act} (t, x))
\end{array}$$
To make the notation more transparent, let us consider the 
commutative diagram
$$\begin{array}{ccccccc}
&&\mathbb G_{\op{m}} \times X& \stackrel{\nu}{\ra} & \mathbb A^1 \times X
&\stackrel{\pi}{\ra} &X\\
&&r\da&&q\da&&p\da\;\;\\
Z&\stackrel{u}{\ra}&\mathbb G_{\op{m}} \times Z& \stackrel{\mu}{\ra} & \mathbb A^1 \times Z
&\stackrel{\omega}{\ra} &Z
\end{array}$$
in which all morphisms except $u$ are the product with suitable identities,
so all squares are cartesian. The morphism $u$ is the unit section
$z\mapsto (1,z)$. 
By weak equivariance, there exists an isomorphism
$
 \nu^\ast \pi^\ast N \cong \nu^\ast 
\tau^\ast \pi^\ast N
$.
On the other hand,  we have
$\kappa^\ast \tau^\ast \pi^\ast N =  p^\ast a^\ast N =0$
 by assumption. The localization sequence for $\kappa$ and $\nu$ thus gives us the third  isomorphism of a chain of morphisms
$$
\pi^\ast N \leftarrow  \nu_!\nu^\ast \pi^\ast N \cong \nu_!\nu^\ast 
\tau^\ast \pi^\ast N\sira
\tau^\ast \pi^\ast N
$$
with adjunction morphisms at the beginning. Clearly all these morphisms pull back to  isomorphisms under $\nu^*$. Applying $q_*$, 
we get a morphism $\alpha: q_*\tau^\ast \pi^\ast N\ra q_*\pi^\ast N$, and  base change shows $\mu^*(\alpha)$ is an isomorphism. On the other hand, the adjunction morphism $\pi^* N\ra \tau_*\tau^*\pi^*N$ also pulls back under $\nu$ to an isomorphism  $\nu^*\pi^* N\sira \nu^*\tau_*\tau^*\pi^*N$, and thus for the induced morphism $\beta: q_*\pi^* N\ra q_*\tau_*\tau^*\pi^*N \sira q_*\tau^*\pi^*N$ by the same argument $\mu^*(\beta)$ is an isomorphism.
We have thus constructed  a morphism 
$$\alpha\circ \beta: q_*\pi^* N\ra q_*\pi^*N$$
with the property, that $\mu^*(\alpha\circ \beta)$ is an isomorphism.
Thus
 $u^*\mu^*(\alpha\circ \beta)$ has to be an isomorphism as well.
Next we show $\omega_*(\alpha\circ \beta)=0$. 
Since this factors through $\omega_*q_*\tau^*\pi^* N
\cong p_*\pi_* \nu_!\nu^*\tau^*\pi^*N
\cong p_*\pi_* \nu_!\nu^*\pi^* N$, it is sufficient to show that the
latter object is zero. For this consider the localization triangle 
$$
\pi_*\nu_!\nu^*\pi^*N\ra \pi_*\pi^*N
\ra \pi_*\kappa_*\kappa^*\pi^*N\ra
\pi_*\nu_!\nu^*\pi^*N[1]
$$
and remark that its second arrow has to be an isomorphism,
so the first term has to be zero. 
However by smooth base change, we get a canonical isomorphism 
$\omega^*p_* N\cong q_*\pi^* N$.
Thus we may apply   \prettyref{lem:conn} below
to our morphism $\alpha\circ \beta$ and deduce that,
since  $u^*\mu^*(\alpha\circ \beta)$ is an isomorphism,
$\omega_*(\alpha\circ \beta)$ has to be an isomorphism, too. 
This however implies $0=\omega_* q_*\pi^* N\cong
\omega_*\omega^*p_* N\cong p_*N$ as claimed.
\end{proof}

\begin{lemma}
\label{lem:conn} 
 Let $Y$ be a variety and $\omega:\mathbb A^n\times Y\ra  Y$ 
the projection. Given $M\in \mathscr{T}(Y)$, the adjunction map is an
isomorphism $\alpha_{M}:M\sira \omega_*\omega^*M$. 
If in addition 
$s:Y\ra \mathbb A^n\times Y$ is any section of the projection,
in other words a morphism with  $\omega\circ s=\op{id}_Y$,  
then for any two objects $M,N\in \mathscr{T}(Y)$
and any morphism $f\in\mathscr{T}_{\mathbb A^n\times Y}(\omega^*M,\omega^*N)$ 
the obvious morphisms form a commutative diagram
$$\begin{array}{cccccc}
s^*\omega^*M &\sira&M\sira  &\omega_*\omega^*M\\
s^*f\da&&&\omega_*f\da \\
s^*\omega^*N &\sira&N\sira  &\omega_*\omega^*N
\end{array}$$
\end{lemma}

\section{Full faithfulness and combinatorial models}
\label{sec:ginzburg}

In this section, we adapt the arguments of Ginzburg \cite{Gi} to the motivic setting. We establish a full faithfulness result which allows to compute morphisms between pure stratified Tate motives in terms of maps between their bigraded motivic cohomology rings. This full faithfulness result will allow us to identify the category $\DMT_{(B)}(G/B)$ in terms of a homotopy category of Soergel modules. 

\begin{Bemerkungl}
\label{eqf}
The full faithfulness result now requires that we work in a motivic triangulated category $\mathscr{T}$ which satisfies both the weight and grading conditions. In particular, the grading condition implies that  $\mathscr{T}_{\op{pt}}(\underline{\op{pt}},\underline{\op{pt}}(p)[q])\neq 0$ only for $p=q=0$, in which case this is a one-dimensional vector space over $\DQ$ generated by the identity morphism of $\underline{\op{pt}}$. For $\mathbb  Q\op{-Modf}^{\mathbb Z\times\DZ}$ the category of finite dimensional $(\mathbb Z\times\DZ)$-graded $\mathbb Q$-vector spaces, we thus get an equivalence of $\DQ$-linear monoidal categories  
$$\mathbb Q\op{-Modf}^{\mathbb Z\times\DZ} 
\sirra \DMT (\op{pt})$$ 
mapping $\mathbb Q$ sitting in bidegree $(p,q)$ to the motive $\underline{\op{pt}}(p)[q]$. 
\end{Bemerkungl}

\begin{definition}
\label{defin:SDF}
Let $(X,\mathcal S)$ be an affinely Whitney--Tate stratified
 variety.  Given stratified mixed Tate motives 
$M,N\in \DMT_{\mathcal S}(X)$
we define the bigraded vector space 
$$
  \overline{\DMT}_{\mathcal S}(M ,N)\pdef  \bigoplus_{(i,j) \in
    \mathbb Z\times \mathbb Z} {\mathscr{T}}_{X} (M, N (i)[j]) 
   $$
This can also be interpreted  as the bigraded vector space corresponding to the motive $\op{fin}_\ast\op{Hom}_X(M,N)$ under the  equivalence \ref{eqf}.  

We consider the bigraded ring 
$$
{\op{H}}X\pdef \overline{\DMT}_{\mathcal S}(\underline X
,\underline X)
$$ 
and the hypercohomology functor 
$$\mathbb H:\DMT_{\mathcal S}(X)\ra \op{Mod}_{{\op{H}}X}:M\mapsto
\overline{\DMT}_{\mathcal S}(\underline X,M). 
$$
\end{definition}

In the following, we will bootstrap Ginzburg's arguments from \cite{Gi} in the setting of motives, and the above bigraded cohomology rings. We fix some terminology to be used throughout the section. For a stratum $X_s$, we denote by $\nu_s:\overline{X_s}\to X$ the inclusion of its closure, by $j_s:X_s\to\overline{X_s}$ the inclusion of the stratum into its closure and by $i_s:\overline{X_s}\setminus X_s\to \overline{X_s}$ the inclusion of the closed complement. We use the notation $\op{fin}$ rather freely, for all sorts of structure morphisms of $k$-varieties, trusting the readers to figure out on their own the variety belonging to the structure morphism.  

We first establish an exact sequence as in \cite[Proposition
3.6]{Gi}. 

\begin{proposition}
\label{prop:gi36}
Let $L,M\in \DMT_{\mathcal{S}}(X)$ be stratified mixed Tate motives, such that $L$ is pointwise $\ast$-pure and $M$ is pointwise $!$-pure. We set $L_s=\nu_s^\ast L$, $M_s=\nu_s^! M$. Then there is an exact sequence
$$
0\to \overline{\DMT}_{\mathcal{S}}(i^\ast L_s,i^!M_s)\to
\overline{\DMT}_{\mathcal{S}}(L_s,M_s)\to
\overline{\DMT}_{\mathcal{S}}(j^\ast L_s,j^\ast M_s)\to 0.
$$
\end{proposition}

\begin{proof}
By an induction on the number of strata, we can assume that $\overline{X}_s=X$. This allows to simplify notation to $j=j_s$ the inclusion of the open stratum and $i=i_s$ it closed complement. 

Consider the internal Hom motive $\op{Hom}(L,M)$ (in motives over $X$) and form the localization  triangle
$$
i_!i^!\op{Hom}(L,M)\ra \op{Hom}(L,M) \ra j_*j^*\op{Hom}(L,M)\to
[1]
$$
By standard isomorphisms it can be transformed to a distinguished triangle
$$
i_*\op{Hom}(i^*L,i^!M)\ra \op{Hom}(L,M) \ra j_*\op{Hom}(j^*L,j^!M)\to[1]
$$ 
Now as in the proof of \prettyref{cor:bswz}  the object $\op{Hom}(L,M)$ and its exceptional pullbacks $i^!\op{Hom}(L,M)$ and $j^!\op{Hom}(L,M)$ are pointwise $!$-pure.   Applying $\op{fin}_*$ will thus lead to a triangle of motives on the point, which are all pure of weight zero, so that the degree-one morphism has to vanish. Applying \prettyref{defin:SDF} this establishes the required short exact sequence. 
\end{proof}

\begin{theorem}[{\bf Full faithfulness of cohomology}] 
\label{thm:ginzburg}
Let $(X,\mathcal S)$ be an affinely Whitney--Tate stratified proper variety and let $L,M\in\DMT_{\mathcal S}(X)$ be pointwise $\ast$-pure  and  pointwise  $!$-pure, respectively. Assume in addition that, for each embedding $j$ of a stratum, $\mathbb H L\ra \mathbb Hj_*j^* L$ is surjective and $\mathbb H j_!j^!M\ra \mathbb H M$ is injective. Then the hypercohomology functor induces a bijection 
$$
  \overline{\DMT}_{\mathcal S}(L ,M)\sira \op{Mod}_{{\op{H}}X}
(\mathbb HL,\mathbb HM).
$$ 
\end{theorem}

\begin{Bemerkungl}
In the above, $\op{Mod}_{{\op{H}}X}$ means the vector space of all homomorphisms of ${\op{H}}X$-modules, ignoring any gradings. Requiring the grading to be respected, we have under the same conditions a bijection 
$$
  \DMT_{\mathcal S}(L ,M)\sira 
  \op{Mod}^{\DZ\times \DZ}_{{\op{H}}X}(\mathbb HL,\mathbb HM)
  $$ 
between morphisms of stratified mixed Tate motives and
morphisms of bigraded ${\op{H}}X$-modules which are homogeneous of
bidegree $(0,0)$. 
We discuss in  \prettyref{rem:CBS} below why the conditions of the theorem are
satisfied for Bott--Samelson  sheaves. In this special case, there is
also an alternative proof comparing dimensions of the homomorphism
spaces involved.  
\end{Bemerkungl}

\begin{proof}
We first note that the morphism is simply given by applying
hypercohomology: an element $f\in
\overline{\DMT}_{\mathcal{S}}(L,M)$ is a map
$f:L\to M(i)[j]$, and the image of $f$ in 
$\op{Mod}_{{\op{H}}X}(\mathbb HL,\mathbb HM)$ is
$\mathbb{H}(f)$. 

The  proof is due to Ginzburg \cite{Gi}, whose arguments we
repeat. Let $u:D\hra X$ be the embedding of an open stratum and
$i:Z\hra X$ the embedding of its closed complement. The proof consists of 
embedding our morphism as  middle vertical in a commutative diagram 
$$\begin{array}{ccccc}
  \overline{\DMT}_{\mathcal S}(i^*L,i^!M)&\hra&\overline{\DMT}_{\mathcal
    S}(L,M)&\sra&\overline{\DMT}_{\mathcal S}(u^*L,u^*M)\\
\da&&\da&&\da\\
\op{Mod}_{{\op{H}}Z}(\mathbb H i^*L,\mathbb H i^!M) &\hra&
\op{Mod}_{{\op{H}}X}(\mathbb HL,\mathbb HM)&\ra&
\op{Mod}_{{\op{H}}D}(\mathbb H u^*L,\mathbb H u^*M) 
\end{array}$$
with the upper row short exact, the lower row left exact, and all vertical maps  given by the corresponding hypercohomology functors. Once this is established, the left vertical is an isomorphism by an induction on the number of strata, for the right vertical this is clear anyhow, and by a diagram
chase we are done. So the problem is to construct the horizontal maps and show the required exactness of the horizontal sequences. The upper sequence is established in \prettyref{prop:gi36}. 

To discuss the lower horizontal, let $c\in {\op{H}}X$ be the class
corresponding to the open cell $D$. This class is given by the following composition, where  $z:\op{pt}\hra X$ denotes the inclusion of any point of the cell $D$ and $d$ denotes the dimension of the cell $D\cong\mathbb{A}^d$:
$$
\underline X 
\ra z_*z^*\underline X\sira z_!z^!\underline
X(d)[2d]\ra 
\underline X(d)[2d]
$$
Note that, given a cohomology class $\gamma:\underline{X}\ra \underline{X}(q)[p]$, its effect on the hypercohomology of a motive $\mathcal F$ can be realized as hypercohomology of the morphism $\gamma\otimes \op{id}:\underline{X}\otimes \mathcal F \ra \underline{X}(q)[p]\otimes \mathcal F$ up to the natural identification  $\underline{X}\otimes \mathcal F\cong\mathcal F$. In our case, the action of the class $c$ on  the hypercohomology of a stratified motive $\mathcal{F}$ is thus induced from the composition 
$$
\mathcal F
\ra z_*z^*\mathcal F\sira z_!z^!\mathcal F (d)[2d]\ra
\mathcal F (d)[2d]
$$
where the first and third maps are units and counits of the respective adjunctions and the middle isomorphism is due to $\mathcal F$ being constant on the open cell $D$. The outer morphisms admit natural factorizations as $\mathcal F \ra u_*u^*\mathcal F \ra z_*z^*\mathcal F$ and $z_!z^!\mathcal F(d)[2d] \ra  u_!u^!\mathcal F(d)[2d]\ra \mathcal F(d)[2d] $, respectively. On the total bigraded hypercohomology, these induce isomorphisms 
$$
\mathbb H (u_*u^*\mathcal F) \sira  \mathbb H (z_*z^*\mathcal F) \textrm{ and }
\mathbb H (z_!z^!\mathcal F(d)[2d]) \sira \mathbb H (u_!u^!\mathcal F(d)[2d]),
$$  
since $X$ is assumed to be proper. Applying this to the motives $M$ and $N$,  we get commutative diagrams
$$\begin{array}{ccccc}
 \mathbb H (u_*u^* L)&\twoheadleftarrow&\mathbb H (L)\\
\wr\da&&c\da&&\\
\mathbb H (u_!u^!L(d)[2d])&\hra&\mathbb H(L(d)[2d])&\sra&\mathbb H (i_*i^* L(d)[2d])
\end{array}
$$
$$
\begin{array}{ccccc}
 \mathbb H (i_!i^!M)&\hra&\mathbb H (M)&\sra&\mathbb H (u_*u^*M)\\
&&c\da&&\wr\da\\
&&\mathbb H (M(d)[2d])&\hookleftarrow&\mathbb H (u_!u^!M(d)[2d])
\end{array}
$$
The upper surjection in the upper diagram and the lower injection in the lower diagram are from the assumptions. The horizontal exact sequences are obtained as in the proof of Proposition~\ref{prop:gi36}, 
using our pointwise purity assumptions on $L$ and $M$. 
For the upper diagram, we need to dualize, which is ok because $X$ is proper and hence $\op{fin}_\ast\cong\op{fin}_!$.

These diagrams lead to isomorphisms 
$(\op{im}c:\mathbb H L \ra  \mathbb HL)\sira \mathbb H u_*u^* L\sira
\mathbb H u^*L$ and  $(\op{im}c:\mathbb H M\ra  \mathbb HM)\sira
\mathbb H u_*u^* M\sira \mathbb H u^*M$. For  the 
lower right horizontal in our diagram from the beginning of the proof
we then just take the map restricting a module
homomorphism to the induced homomorphism on $\op{im}c$. 

We have to check that the right square commutes. The map
$\overline{\DMT}_{\mathcal{S}}(L,M)\to
\overline{\DMT}_{\mathcal{S}}(u^\ast L,u^\ast M)$
comes from $u^\ast$-restriction  of the inner Hom. Since the map $c$
in the diagram is similarly defined via the restriction functors, the
right square of the diagram commutes.
The lower left horizontal in our diagram from the
beginning of the proof comes from the natural
morphisms $L\ra i_*i^*L$ and $i_!i^!M\ra M$
and is an injection, since we have 
$\mathbb H L\sra \mathbb H i_*i^* L$ and $\mathbb H
i_!i^! M\hra\mathbb H M$ by the above.
The composition in the  lower horizontal is clearly zero. The only thing left
to show is that in the middle each element in the kernel also
belongs to the image. Now if $f:\mathbb H L\ra \mathbb H M$
goes to zero, it will obviously factor as
$\mathbb H L\sra (\op{cok}c)\ra (\op{ker}c)\subset \mathbb H M$.
But the left
diagram above gives us a natural isomorphism
$(\op{cok}c)\sira \mathbb H
i_*i^* L$ and the right diagram above shows that  $(\op{ker}c)$ 
is the image
of $\mathbb H i_!i^! M\hra \mathbb H M$.
Thus $ f$ will actually come from some
$\tilde f:\mathbb H
i_*i^* L\ra \mathbb H i_!i^! M$ as claimed.
\end{proof}

\begin{theorem}\label{thm:ffbs} 
  Let $G\supset P\supset B$ be a reductive algebraic group over $k$ with
a choice of Borel subgroup $B$ and parabolic subgroup $P$. Then on the
heart of the weight structure
$
\DMT^{\op{bs}}_{(B)}(G/P)= \DMT_{(B)}(G/P)_{w=0}
$ from \prettyref{cor:bsgen} the hypercohomology functor $\mathbb H$ 
from \prettyref{defin:SDF} restricts to
 a fully faithful functor 
$$
\mathbb H:\DMT^{\op{bs}}_{(B)}(G/P)\stackrel{\sim}{\hra}
\op{Mod}_{\op{H}(G/P)}^{\DZ\times\DZ}
$$
\end{theorem}

\begin{Bemerkungl}\label{FGH} 
 Since $\op{fin}_!=\op{fin}_*$ has to preserve weights, 
it is clear that the modules in the image of our functor
will only live in bidegrees $(2j,j)$. If we just keep the first, i.e.,
the cohomological grading, the category of graded modules 
over the cohomology ring ${\op{H}}^*(G/P)$ of the flag variety 
forming the essential image of our functor will be denoted
${\op{H}}^*(G/P)\op{-SMod}^\DZ_{\op{ev}} $. It consists of the modules with
even grading in a category of graded modules sometimes called
``Soergel modules''. 
\end{Bemerkungl}

\begin{proof}
 We apply \prettyref{thm:ginzburg} on the full faithfulness of
 hypercohomology and have to check that the conditions needed are
 satisfied. We already know from \prettyref{lem:bswz} or alternatively
 \prettyref{prop:wep}  
that Bott--Samelson sheaves are pointwise pure. The remaining conditions
are easily deduced from \prettyref{rem:CBS} below.  
\end{proof}

\begin{proposition}
\label{prop:CBS}
   Let $X$ be a proper variety and let
$M\in\mathscr{T}^c(X)_{w=0}$ be pure. 
Let $v:V\hra X$ be the embedding of
an open subset and suppose there is an
action of $\mathbb G_{\op{m}}$ on $V$ contracting $V$ to a fixed point
$x\in V$, for which $M$ is weakly equivariant.   
Then for the inclusion $i:x\hra X$ the obvious map is a surjection
$$\mathbb HM\sra \mathbb H i_*i^*M.$$
\end{proposition}

\begin{proof}
  This is due to Ginzburg \cite{Gi}, whose arguments we repeat. 
By \prettyref{prop:wep} the contraction induces
an isomorphism
$\op{fin}_*v^* M\sira  i^*M
$
and both sides are pure of weight zero. 
If we now let $r$ be the embedding of
the complement of $V$, we get a distinguished triangle  
$$\op{fin}_* r_!r^! M\ra  \op{fin}_* M\ra
 \op{fin}_* v_*v^* M\ra \op{fin}_* r_!r^! M[1].$$
Here the degree one morphism  has to vanish, since both $r^!$ and
$r_!=r_*$  never make weights smaller, so we get a short exact sequence 
$$0\ra\mathbb H r_!r^! M\hra  \mathbb H M\sra
 \mathbb H  v_*v^* M\ra 0$$
and in particular a surjection $\mathbb H M\sra
 \mathbb H  i_*i^* M$.
\end{proof}

\begin{remark}
\label{rem:CBS}
  If in addition the inclusion of the point $x$ factors over the inclusion
of an affine space $j:D\hra X$ and $j^*M\in \DMT(D)$ is
mixed Tate, then our surjection factors as
$\mathbb H M\ra  \mathbb H j_*j^*M\sira
\mathbb H i_*i^*M$ and thus the first map has to be
a surjection as well. 
   Dual arguments show  the injectivity $\mathbb H
  i_!i^!M\hra \mathbb H M$ 
and of  $\mathbb H
  j_!j^!M\hra \mathbb H M$
under the dual assumptions. 
\end{remark}

\begin{remark}
Similar results hold under more general assumptions. For instance, if
$k$ is a number field and $\mathscr{T}=\Bmot$ or
$\mathscr{T}=\mathbf{DA}_{\et}$, then the grading condition is not
satisfied. Nevertheless, a full faithfulness result as above remains
true. It expresses morphisms between stratified mixed Tate motives in
terms of morphisms between the associated motivic cohomology
rings. However, the actual description of the motivic cohomology of
the flag variety is more complicated. It combines the motivic
cohomology of the base field (which is quite nontrivial) with the cell
structure information of the flag variety. Since it is not clear if
full faithfulness in such more general situations can be useful, we
chose not to spell out the details. 
\end{remark}

\section{Tilting for motives}
\label{sec:TE} 

\begin{Bemerkungl}
\label{conv9}
In this section, we require that the motivic triangulated category
$\mathscr{T}$ is one of the following: $\mathscr{T}=\Bmot$,
$\mathscr{T}=\mathbf{DA}_{\et}$ or
$\mathscr{T}=\mathcal{E}_{\op{GrH}}$. All these categories satisfy 
both the weight and grading condition. However, for the tilting results in this
section, we need additional information on how the motivic
triangulated categories are constructed. All the above categories are
constructed as suitable localizations of abelian categories of
(symmetric spectra in) complexes of sheaves on a site. In this
situation, we can apply the tilting result \prettyref{prop:rvt}. This
is the only place of the paper where we need such explicit information
on the construction, and can not make do with the axiomatics of
motivic triangulated categories. Sorry.
\end{Bemerkungl}

\begin{theorem}
\label{thm:Twer} 
Let $(X,\mathcal{S})$ be a Whitney--Tate affinely 
stratified variety. Assume the setting laid out in \ref{conv9}. Assume
furthermore that all objects of $\DMT_{\mathcal{S}}(X)_{w=0}$ are
pointwise pure. Then the tilting functor, cf. \prettyref{prop:rvt},
induces an equivalence 
$$
\op{Hot}^{\op{b}}(\DMT_{\mathcal{S}}(X)_{w=0})
\stackrel{\approx}{\longrightarrow} 
\DMT_{\mathcal{S}}(X)
$$
between the category of stratified mixed Tate
motives on $X$ and the bounded homotopy category of the heart of the
weight structure.
\end{theorem}

\begin{Bemerkungl}
Remark that by \cite[Proposition 1.7(6)]{bondarko:imrn}
two  weight structures on an idempotent complete triangulated category
with the same heart are equal, if this heart already generates the
whole triangulated category in question.
Now  the obvious embedding $\DMT_{\mathcal{S}}(X)_{w=0}
\hra \op{Hot}^{\op{b}}(\DMT_{\mathcal{S}}(X)_{w=0})$ as complexes
concentrated in degree zero induces an equivalence with the heart
of the obvious weight structure on $\op{Hot}^{\op{b}}$. On the other hand,
its composition with the equivalence of the theorem also
induces an equivalence with the heart of the motivic weight structure
on $\DMT$, since by construction this composition is isomorphic to
the embedding of the heart of the weight structure. Thus under the
equivalence of the theorem the obvious weight structure on
$\op{Hot}^{\op{b}}$ coincides with the motivic weight structure on
$\DMT$. 
\end{Bemerkungl}

\begin{proof}
This is a special case of the general tilting equivalence 
from \prettyref{prop:rvt}.  
Repeating the proof of \prettyref{cor:bswz}, for any two
 pointwise pure stratified
mixed Tate motives $M,N\in \DMT_{\mathcal{S}}(X)_{w=0}$, 
we deduce $\DMT_{\mathcal{S}}(M,N[a])=0$
for $a\neq 0$ from the grading condition
and thus by \ref{eqf} there are no nonzero
morphisms between objects of different weight in $\DMT(\op{pt})$. 
Now, in the situation fixed in \ref{conv9}, $\mathscr{T}(X)$ is
constructed from $\op{Der}(\op{Sh}_{\tau}(\op{Sm}/S,\mathbb{Q}))$
by $\Ao$-localization, stabilization via symmetric spectra and possibly a
further Bousfield localization at $\op{H}_{\Be}$ or
$\mathscr{E}_{\op{GrH}}$. In particular,  $\mathscr{T}(X)$ can be 
embedded as a full subcategory of  
the derived category of an abelian category: the abelian category is
the one of symmetric sequences in
$\op{Sh}_{\tau}(\op{Sm}/S,\mathbb{Q})$. 
Finally, $\DMT_{\mathcal{S}}(X)$ embeds by
definition as full subcategory of $\mathscr{T}(X)$. 
 Using this embedding, it is possible
to choose injective resolutions for
the objects of $\DMT_{\mathcal{S}}(X)_{w=0}$. These form a tilting
collection satisfying all the conditions necessary to apply
\prettyref{prop:rvt}. 
This implies the existence of a fully faithful functor  
$$
\op{Hot}^{\op{b}}(\DMT_{\mathcal{S}}(X)_{w=0})
\stackrel{\sim}{\hookrightarrow}
\DMT_{\mathcal{S}}(X)
$$
The heart of the weight structure on $\DMT_{\mathcal{S}}(X)$ generates
the category, therefore the functor is also essentially surjective. 
\end{proof}

\begin{corollary}
\label{cor:ttb}   
  Let $G\supset P\supset B$ be a split reductive algebraic group over the field $k$, with a choice of Borel subgroup $B$ and parabolic subgroup $P$,  and let $Y$ be a paraboloid $B$-variety. Then the tilting functor of \prettyref{prop:rvt} provides an equivalence of categories 
$$
\op{Hot}^{\op{b}}(\DMT_{(B)}(Y)_{w=0})\sirra  \DMT_{(B)}(Y)
$$
between the bounded homotopy category of the additive category of pure Bruhat--Tate sheaves and the triangulated category of all Bruhat--Tate sheaves.  
For $Y=G/P$, we obtain an equivalence of triangulated categories 
$$\op{Hot}^{\op{b}}(\op{H}^*(G/P)\op{-SModf}^{\mathbb Z}_{\op{ev}})
\sirra
\DMT_{(B)}(G/P)$$ 
between the bounded homotopy category of even Soergel modules and
stratified mixed Tate motives over $G/P$. 
\end{corollary}

\begin{proof}
By \prettyref{cor:bsgen}
we have
$
\DMT^{\op{bs}}_{(B)}(G/P)= \DMT_{(B)}(G/P)_{w=0}
$ and by \prettyref{lem:bswz} all objects of this category are
pointwise pure. The same statements follow easily for any paraboloid 
$B$-variety, and thus the first equivalence is a special case of 
\prettyref{thm:Twer}.  The second equivalence follows using
the faithfulness \prettyref{thm:ffbs} in conjunction with 
the definition of Soergel modules from \ref{FGH}. 
\end{proof}


\section{Perverse Tate motives}
\label{sec:tstructure}

In this section, we describe a t-structure on the category
$\DMT_{\mathcal{S}}(X)$ of stratified mixed Tate motives, for
$(X,\mathcal{S})$ an affinely Whitney--Tate stratified variety.  
The t-structure is obtained via the BBD-glueing formalism \cite{BBD}
from the t-structure on mixed Tate motives $\DMT(k)$, which
exists for  base fields satisfying the Beilinson--Soul{\'e} vanishing
conjectures. The heart of the t-structure is an abelian category of
perverse mixed Tate 
motives. In the next section, we will show that the perverse mixed
Tate motives provide a grading on category $\mathcal{O}$.

\begin{Bemerkungl}\label{BeiS} 
In this section, we assume that the motivic triangulated category
satisfies the grading condition. Alternatively, working with \'etale or
Beilinson motives, the results also work if we assume that the ground
field $k$ satisfies  the Beilinson--Soul\'e vanishing conjectures.
\end{Bemerkungl}

\begin{Bemerkungl}
Using the work of Levine \cite{levine:nato}, this assumption implies
that the categories $\DMT(X_s)$ of mixed Tate motives on the strata
$X_s\cong\mathbb{A}^{n_s}$ have non-degenerate t-structures. For a
more detailed recollection of the motivic t-structures and abelian
categories of mixed Tate motives, see \prettyref{sec:mtm}. 
\end{Bemerkungl}

\begin{theorem}
\label{thm:perv}
Let  $(X,\mathcal S)$ be an affinely Whitney--Tate stratified 
variety. For any perversity function $p:\mathcal S\ra \DZ$ 
the following subcategories define  a  t-structure on
$\DMT_{\mathcal S}(X)$:
$$
\begin{array}{l}
\DMT_{\mathcal{S}}(X)^{\leq 0} \pdef \left\{
M\mid j_s^\ast M \in \DMT(X_s)^{\leq p(s)}\text{ for all strata
}s\in\mathcal S\right\}

\\[2mm]
\DMT_{\mathcal{S}}(X)^{\geq 0} \pdef \left\{
M\mid j_s^! M \in \DMT(X_s)^{\geq p(s)}\text{ for all strata
}s\in\mathcal S\right\}
\end{array}
$$
\end{theorem}

\begin{proof}
The proof proceeds by induction on the number of strata. For the base case, we can use the t-structure given by \prettyref{thm:levbs}. 

Otherwise choose an open stratum
$j:U\hra X$ and its closed complement $i:Z\hra X$.
By inductive assumption, we have a non-degenerate t-structure on
$\DMT_{\mathcal{S}}(\overline{X_s}\setminus X_s)$. 
On the open stratum $U$, we have a t-structure on
$\DMT_{\mathcal{S}}(Z)$, again from \prettyref{thm:levbs}. 

We want to glue these two t-structures to obtain  a t-structure on
$\DMT_{\mathcal S}(X) $ with 
$$
\begin{array}{l}
\DMT_{\mathcal{S}}(X)^{\leq 0}\pdef\left\{M\mid i^\ast M\in
  \DMT(Z)^{\leq 0}, \;
  j_t^\ast M\in \DMT(X_t)^{\leq p(t)}\right\}

\\[2mm]
\DMT_{\mathcal{S}}(X)^{\geq 0}\pdef\left\{M\mid i^!M\in
  \DMT(Z)^{\geq 0}, \;
  j_t^! M\in \DMT(X_t)^{\geq p(t)}\right\}  
\end{array}$$
The claim that this is indeed  a non-degenerate t-structure on 
$\DMT_{\mathcal{S}}(\overline{X_s})$ is a consequence of
\cite[Theorem 1.4.10]{BBD} once we verify the axioms
\cite[1.4.3]{BBD}. 

Some proofs of parts of the axioms are deferred to the following
subsection. The first two axioms, 1.4.3.1 and
1.4.3.2, are satisfied by the assumption and \prettyref{prop:adj}. 
The axioms 1.4.3.3 and 1.4.3.5 are easy to see, using basic properties
of the six-functor formalism for motives. With all the functors
restricting to the subcategories $\DMT_{\mathcal{S}}$, the
localization sequence of the 
motivic triangulated category $\mathscr{T}$ also restricts to the
triangulated subcategories $\DMT_{\mathcal{S}}$, hence we also have
axiom 1.4.3.4.

It is then clear that this t-structure can also
be described by the non-inductive formulas given in the proposition.
\end{proof}

\begin{Bemerkungl}
 We are only interested in the case of the so-called middle perversity
 given by $p(s)=-\op{dim}X_s$. For this perversity, we denote the
 heart of the corresponding t-structure by $$\PMT_{\mathcal
   S}(X) $$ and call its objects {\bf perverse mixed Tate  motives on $X$.}
\end{Bemerkungl}

\begin{proposition}
\label{prop:adj}
Let $(X,\mathcal{S})$ be an affinely Whitney--Tate stratified
variety. Then we have the following:

\begin{enumerate}
\item
The functor $j^\ast:\mathscr{T}(\overline{X_s})\hookrightarrow
\mathscr{T}(X_s)$
restricts to a functor
$$
j^\ast:\DMT_{\mathcal{S}}(\overline{X_s})\rightarrow
\DMT(X_s).
$$
\item The functors 
$i^\ast,i^!:\mathscr{T}(\overline{X_s})\hookrightarrow
\mathscr{T}(\overline{X_s}\setminus X_s)$
restrict to functors
$$
i^\ast,i^!:\DMT_{\mathcal{S}}(\overline{X_s})\rightarrow
\DMT_{\mathcal{S}}(\overline{X_s}\setminus X_s).
$$
\end{enumerate}
In particular, the adjunction conditions \cite[1.4.3.1 and
1.4.3.2]{BBD} are satisfied. 
\end{proposition}

\begin{proof}
(1) By the definition of
$\DMT_{\mathcal{S}}(\overline{X_s})$ as triangulated
subcategory generated by the images of $i_\ast=i_!$, $j_\ast$ and
$j_!$ and the fact that $j^\ast$ is a triangulated functor, it
suffices to prove the assertion for these generators. For elements of
the form $j_\ast M$ and $j_!M$, the claim follows from the well-known
identifications $j^\ast j_\ast\cong \op{id}\cong j^\ast j_!$ and
$i^\ast i_\ast\cong \op{id}\cong i^! i_\ast$.  
For $M\in \DMT_{\mathcal{S}}(Z)$, we have $j^\ast
i_\ast=0$. Hence the claim follows.

(2) Now let $M\in \DMT(U)$. We want to prove that the images
of the functors $i^\ast j_\ast$ and $i^!j_!$ lie in
$\DMT_{\mathcal{S}}(\overline{X_s}\setminus X_s)$. Since
$i^\ast j_\ast$ is dual to $i^!j_!$ and the motivic duality restricts
to $\DMT_{\mathcal{S}}(\overline{X_s}\setminus X_s)$, it
suffices to prove one of the assertions. 

We prove by induction that for each stratum $X_t$ in
$\overline{X_s}\setminus X_s$ with inclusion
$i_t:\overline{X_t}\hookrightarrow \overline{X_s}$ we have
$i_t^\ast j_\ast M\in
\DMT_{\mathcal{S}}(\overline{X_t})$. This will 
in particular prove the claim. For a closed stratum $X_t$ in
$\overline{X_s}$, this follows from the assumption that
$(X,\mathcal{S})$ is Whitney--Tate. For the
inductive step, we use the  localization
sequence of the motivic triangulated category $\mathscr{T}$ on
$\mathscr{T}(\overline{X_t})$. We denote by $i_Y:Y=\overline{X_t}\setminus
X_t\hookrightarrow \overline{X_t}$ the closed immersion and by
$j_t:X_t\hookrightarrow \overline{X_t}$ its open complement. The
inductive assumption is  that the claim is true for $Y$,
i.e., $i_Y^\ast i_t^\ast j_\ast M\in  \DMT_{\mathcal{S}}(Y)$. Then, by
assumption that $(X,\mathcal{S})$ is Whitney--Tate again, we also have
$j_t^\ast i_t^\ast j_\ast M\in   
\DMT(X_t)$. The localization sequence decomposes
$i_t^\ast j_\ast M$ as 
$$
(j_t)_! j_t^\ast i_t^\ast j_\ast M\rightarrow i_t^\ast j_\ast
M\rightarrow (i_Y)_\ast i_Y^\ast i_t^\ast j_\ast M\rightarrow 
(j_t)_! j_t^\ast i_t^\ast j_\ast M[1]. 
$$
By what was said above, the first and third term are in
$\DMT_{\mathcal{S}}(\overline{X_t})$, which proves the claim.

Finally, the Axioms 1.4.3.1 and 1.4.3.2 follow since all the six functors
restrict to the categories $\DMT_{\mathcal{S}}$. Since these are full
subcategories, the corresponding adjunctions between the functors also
restrict to $\DMT_{\mathcal{S}}$. 
\end{proof}

We list some of the further consequences of the glueing formalism for
t-structures from \cite[Section 1.4]{BBD}.

First of all, we note that there are modified versions of the six
functors. For a stratified scheme $X$, a stratum $X_s$ and the
inclusions $i:\overline{X_s}\setminus X_s\hookrightarrow
\overline{X_s}$ and $j:X_s\hookrightarrow \overline{X_s}$, we can
define the following functors:
$$
{}^pj_!,
{}^pj_\ast:\op{MT}(X_s)\leftrightarrows
\PMT_{\mathcal{S}}(\overline{X_s}): {}^pj^!={}^pj^\ast
$$
$$
{}^pi_!={}^pi_\ast:\PMT_{\mathcal{S}}(\overline{X_s}\setminus
X_s)\leftrightarrows \PMT_{\mathcal{S}}(\overline{X_s}): 
{}^pi^!,{}^pi^\ast. 
$$
These form adjunctions ${}^pj_!\dashv {}^pj^\ast\dashv {}^pj_\ast$ and
${}^pi^\ast\dashv {}^pi_\ast\dashv {}^pi^!$, cf. \cite[Proposition
1.4.16]{BBD}. 
There is also a modified analogue of the localization sequences: for
each perverse mixed Tate motive $M\in
\PMT_{\mathcal{S}}(\overline{X_w})$, there are by
\cite[Lemma 1.4.19]{BBD} exact sequences
$$
0\rightarrow {}^pi_\ast \mathcal{H}^{-1}i^\ast M\rightarrow {}^pj_!{}^pj^\ast
M\rightarrow M\rightarrow {}^pi_\ast{}^pi^\ast M\rightarrow 0
$$
$$
0\rightarrow {}^pi_\ast{}^pi^!M\rightarrow M\rightarrow
{}^pj_\ast{}^pj^\ast M\rightarrow {}^pi_\ast \mathcal{H}^1i^!M\rightarrow 0.
$$

As in \cite[Definition 1.4.22]{BBD}, we can define a functor
``intermediate extension'' as
$$
j_{!\ast}:\op{MT}(X_s)\rightarrow
\PMT_{\mathcal{S}}(\overline{X_s}): M\mapsto \op{Im}\left({}^p
  j_!M\rightarrow {}^pj_\ast M\right). 
$$
Note that an intermediate extension of Chow motives has already been
considered in \cite{wildeshaus:intermediate} also in situations where
the motivic t-structure is not available.

Finally, \cite[Proposition 1.4.26]{BBD} characterizes the simple
perverse mixed Tate motives in
$\PMT_{\mathcal{S}}(\overline{X_s})$ as those of
the form ${}^pi_\ast M$ for $M$ a simple perverse mixed Tate motive in  
$\PMT_{\mathcal{S}}(\overline{X_s}\setminus X_s)$
and those of the form $j_{!\ast}\mathbb{Q}(a)[-p(s)]$, $a\in\mathbb{Z}$.
The representation-theoretic significance of these objects, the
intersection complexes, will be discussed in the next section.

\section{Motivic graded versions}
\label{sec:cato} 

In this section, we discuss graded versions of category $\mathcal{O}$
arising from motivic triangulated categories. We consider two
versions, an $\ell$-adic and a Hodge version. 

\begin{Bemerkungl}\label{Real-ladic} 
For the $\ell$-adic version, let $k=\mathbb{F}_q$ be a finite field, 
and consider the motivic triangulated categories $\mathscr{T}=\Bmot$
or $\mathscr{T}=\mathbf{DA}_{\et}$ over $\op{Sch}/k$. 
Let us consider a prime $\ell$ different from the characteristic of
  $k$. For a $k$-variety $X$ we consider
 the derived category $\op{Der}(X\times_k\bar k;\DQ_\ell)$
of the category of $\ell$-adic sheaves on $X\times_k\bar k$. 
The $\ell$-adic realization of \cite{cisinski:deglise} followed by
pulling back to the geometric situation gives  triangulated functors
$$\op{Real}_\ell: \Bmotc(X;\DQ_\ell)\ra \op{Der}^{\op{b}}(X\times_k\bar
k;\DQ_\ell)$$ 
compatible with all six functors of Grothendieck, where we take
motives with $\DQ_\ell$-coefficients for better compatibility.  
For an affinely Whitney--Tate stratified variety  $(X,\mathcal S)$,
we denote by $\op{Der}_{\mathcal S}(X\times_k\bar
k;\DQ_\ell)\subset \op{Der}^{\op{b}}(X\times_k\bar k;\DQ_\ell)$
the full triangulated subcategory of all complexes whose
restrictions to all strata are constant of finite rank.
 Then the above realizations 
induce triangulated functors 
$$\op{Real}_\ell: \DMT_{\mathcal S}(X;\DQ_\ell)\ra \op{Der}_{\mathcal
  S}(X\times_k\bar k;\DQ_\ell)$$
Any choice of an isomorphism $\op{Real}_\ell(\underline{\op{pt}}(1))\cong
\op{Real}_\ell(\underline{\op{pt}})$ leads to natural isomorphisms 
$\op{Real}_\ell\mathcal F(n)\sira \op{Real}_\ell\mathcal F$. 
\end{Bemerkungl}

\begin{Bemerkungl}\label{Real-hodge} 
For the Hodge version, let $k=\mathbb{C}$ 
and consider the motivic triangulated categories
$\mathscr{T}=\mathcal{E}_{\op{GrH}}$ over $\op{Sch}/k$. 
For a  complex variety $X$ we consider
 the derived category $\op{Der}(X;\mathbb{C})$
of the category of sheaves of $\mathbb{C}$-vector spaces on $X$. 
The Hodge realization of \cite{drew} gives  triangulated functors
$$\op{Real}_H: \mathscr{T}(X;\mathbb{C})\ra
\op{Der}^{\op{b}}(X;\mathbb{C})$$  
compatible with all six functors of Grothendieck.
For an affinely Whitney--Tate stratified variety  $(X,\mathcal S)$,
we denote by $\op{Der}_{\mathcal S}(X;\mathbb{C})\subset
\op{Der}^{\op{b}}(X;\mathbb{C})$ 
the full triangulated subcategory of all complexes whose
restrictions to all strata are constant of finite rank.
 Then the above realizations 
induce triangulated functors 
$$\op{Real}_H: \DMT_{\mathcal S}(X;\mathbb{C})\ra \op{Der}_{\mathcal
  S}(X;\mathbb{C})$$
Any choice of an isomorphism $\op{Real}_H(\underline{\op{pt}}(1))\cong
\op{Real}_H(\underline{\op{pt}})$ leads to natural isomorphisms 
$\op{Real}_H\mathcal F(n)\sira \op{Real}_H\mathcal F$. 
\end{Bemerkungl}

Note that in both these cases, the weight and grading condition on the motivic triangulated category $\mathscr{T}$ are satisfied, so that all the previously established results are applicable. 

\begin{theorem}
\label{thm:fulg} 
Let $(X,\mathcal S)$ be an affinely Whitney--Tate stratified variety.
\begin{enumerate}
\item 
In the situation of \ref{Real-ladic},  for
any  $\mathcal F, \mathcal G\in \DMT_{\mathcal S}(X;\DQ_\ell)$, 
the realization functor together with the  isomorphisms in \ref{Real-ladic}
above leads to isomorphisms
$$ \bigoplus_{n\in\DZ}\DMT_{\mathcal S}(\mathcal F, \mathcal G(n))
\sira \op{Der}_{\mathcal S}(\op{Real}_\ell\mathcal F, \op{Real}_\ell\mathcal
G).$$
\item In the situation of \ref{Real-hodge},  for
any  $\mathcal F, \mathcal G\in \DMT_{\mathcal S}(X;\mathbb{C})$, 
the realization functor together with the  isomorphisms in \ref{Real-hodge}
above leads to isomorphisms
$$ \bigoplus_{n\in\DZ}\DMT_{\mathcal S}(\mathcal F, \mathcal G(n))
\sira \op{Der}_{\mathcal S}(\op{Real}_H\mathcal F, \op{Real}_H\mathcal G)$$
\end{enumerate}
\end{theorem}

\begin{proof}
We give the proof of (1), the proof of (2) is similar.
We know from \prettyref{sec:dmt} that $\DMT_{\mathcal S}$
is generated as a triangulated category  by the 
shifted twisted costandard objects
$j_{s!}\underline{X}_s(n)$ as well as by the shifted twisted standard objects 
$j_{s*}\underline{X}_s(m)$. 
By devissage, it is sufficient to check the claim for
$\mathcal F$ costandard and $\mathcal G$ standard. In this case
however, we can use base change to switch to the case of a single
stratum,which  follows from \ref{eqf}: the identification of
morphisms in $\DMT(\underline{\op{pt}})$ with Adams eigenspaces of
Quillen $\op{K}$-theory implies that
$\Bmot(\underline{\op{pt}},\underline{\op{pt}}(p)[q])\neq 0$ 
only for $p=q=0$, in which case this is a one-dimensional 
vector space over $\DQ$ generated by the identity morphism of
$\underline{\op{pt}}$, and then the claim follows from homotopy invariance.
\end{proof}
\begin{Bemerkungl}\label{raed} 
Let $(X,\mathcal S)$ be an 
affinely Whitney--Tate stratified variety.
By compatibility with the six functors, the realization from 
\ref{Real-ladic} or \ref{Real-hodge} induces an exact functor between
 the corresponding categories of perverse sheaves
$$\op{Real}_\ell: \PMT_{\mathcal S}(X;\DQ_\ell)\ra \op{Perv}_{\mathcal
  S}(X\times_k\bar k;\DQ_\ell)$$
$$\op{Real}_H: \PMT_{\mathcal S}(X;\mathbb{C})\ra \op{Perv}_{\mathcal
  S}(X;\mathbb{C})$$

Clearly, a stratified mixed Tate motive in $\DMT_{\mathcal S}(X)$ is perverse if and only if its realization is perverse. We deduce from 
\cite[4.1.3]{BBD} that the costandard objects $\Delta_s\pdef
j_{s!}\underline{X}_s[\op{dim}X_s]$  
as well as  the standard objects 
$\nabla_s\pdef j_{s*}\underline{X}_s[\op{dim}X_s]$ are actually 
perverse motives, i.e., they belong to $\PMT_{\mathcal S}(X)$. 
As an aside, let us remark that the last statement even follows with
$\DQ$-coefficients. 
\end{Bemerkungl}

\begin{remark}
It would be much more satisfying to have a ``motivic'' proof that the
standard and costandard objects are perverse, without having to resort
to checking it on \'etale realization. However, this would require a
version of Artin vanishing in the motivic setting, which at the moment
does not seem to be known. We thank Rahbar Virk for discussions about
this point. Actually, in the case of
$\mathscr{T}=\mathcal{E}_{\op{GrH}}$, it might actually be possible to
translate the statements known in complex geometry to the ``motivic
setting'', but we have not checked that. 
\end{remark}

\begin{lemma}
\label{lem:cond4}
Assume the situation in \ref{Real-ladic} or \ref{Real-hodge}, and let
$(X,\mathcal{S})$ be 
an affinely Whitney--Tate stratified variety. Consider the category
$\PMT_{\mathcal{S}}(X)$, i.e., we take perverse motives for the
middle perversity. Let $j:U\to X$ be an open stratum of dimension $d$. 
\begin{enumerate}
\item The object $j_{!\ast}\mathbb{Q}[d]$ is simple.
\item The object $j_!\mathbb{Q}[d]$ is the projective cover of
  $j_{!\ast}\mathbb{Q}[d]$. 
\item The object $j_\ast\mathbb{Q}[d]$ is the injective hull of
  $j_{!\ast}\mathbb{Q}[d]$. 
\end{enumerate}
\end{lemma}

\begin{proof}
We have seen in \ref{raed} above (using $\ell$-adic realization) that
all the objects appearing are indeed perverse motives,
i.e., that $j_{!\ast}\mathbb{Q}[d]$, $j_!\mathbb{Q}[d]$ and
$j_\ast\mathbb{Q}[d]$ are in $\PMT_{\mathcal{S}}(X)$. 

As mentioned earlier, (1) is a consequence of \cite[Proposition
1.4.26]{BBD}. The statements (2) and (3) are dual, we only prove (2). 

We first note that $\mathbb{Q}[d]$ is a projective object in
$\DMT(k)$: by assumption, we can identify $\DMT(k)$ with the
bounded derived category of graded $\mathbb{Q}$-vector spaces (with
homogeneous maps). The category $\op{MT}(k)[d]$ then
consists of graded vector spaces, considered as complexes concentrated
in degree $d$. In that case, projectivity of $\mathbb{Q}$ is obvious. 

Now we discuss projectivity of $j_!\mathbb{Q}[d]$. We are grateful to
Rahbar Virk for pointing out the following argument. To prove
projectivity it suffices to show vanishing of
$\op{Ext}^1(j_!\mathbb{Q}[d],M)=0$, where 
$\op{Ext}^1$ is to be interpreted as morphisms in the derived category
${\op{Der}_{\PMT_{\mathcal{S}}(X)}}(j_!\mathbb{Q}[d],M[1])$. The
latter can be identified with Yoneda $\op{Ext}^1$, and via
\cite[Corollary 1.1.10, Theorem 1.3.6]{BBD} with $\op{Ext}^1$ in the
category $\DMT_{\mathcal{S}}(X)$. Using the adjunctions of the
six-functor formalism and the vanishing of $j^\ast i_!$ and $i^\ast
j_!$, we find $\op{Ext}^1(j_!\mathbb{Q}[d],(i_s)_{!*}\mathbb{Q}[1])=0$
for any stratum other than the open. On the open stratum, the
vanishing of the $\op{Ext}^1$ follows from $\mathbb{Q}[d]$ being
projective in $\DMT(k)$ and homotopy invariance. 

To see that $j_!\mathbb{Q}[d]\to j_{!\ast}\mathbb{Q}[d]$ is the
projective cover, we also use an adjunction argument. It follows from
short exact sequences before \cite[Corollaire 1.4.24]{BBD} that the
kernel of the surjection $j_!\mathbb{Q}[d]\to j_{!\ast}\mathbb{Q}[d]$
is supported on the complement of $U$. Any submodule $M$ of
$j_!\mathbb{Q}[d]$ whose sum with the kernel equals $j_!\mathbb{Q}[d]$
then has to be $j_!\mathbb{Q}[d]$, so $j_!\mathbb{Q}[d]$ is in fact
the projective cover of $j_{!\ast}\mathbb{Q}[d]$. 
\end{proof}

\begin{proposition}
\label{prop:PrOn}
Assume the situation in \ref{Real-ladic} or \ref{Real-hodge}, and let
$(X,\mathcal S)$ be an  
affinely Whitney--Tate stratified variety. Then the abelian category 
 $\PMT_{\mathcal S} (X)$ has finite homological dimension
and  enough  projective objects and each of
 those has a finite filtration  with  subquotients of the form
 $\Delta_s (\nu)$ for $ s\in\mathcal S$    and $\nu \in
 \mathbb Z$. Similarly,  it has  enough  injective objects and each of
 those has a finite filtration  with  subquotients of the form
 $\nabla_s (\nu)$.  
\end{proposition}

\begin{proof}
We want to apply \cite[Theorem 3.2.1]{BGSo}. We note that a version of
this results is true where (2) is replaced by the requirement that the
partial order in (3) satisfies the descending chain condition. This is
necessary because condition (2) is not satisfied in our situtation:
for each stratum $X_s$ of dimension $d_s$ with $j:X_s\to\overline{X_s}$
and $i:\overline{X_s} \to X$, we have that all $i_\ast\circ
j_{!\ast}\mathbb{Q}(a)[d_s]$, $a\in\mathbb{Z}$, are simple.  

The strengthened condition (3) is then still true, the partial order
is given by the inclusion of support of $M\in\PMT_{(B)}(Y)$ and
the descending chain condition follows since there are only finitely
many strata in $Y$. 

Condition (1) is satisfied, i.e., $\PMT_{(B)}(Y)$ is an artinian
category, every object has finite length: the functors ${}^pi_\ast$
etc. are defined by applying $i_\ast$ and then truncating. Therefore,
these functors preserve finite length of 
objects. We can use the exact sequences of \cite[Lemme 1.4.19]{BBD} to
inductively reduce the finite length assertion to artinianness of
$\op{MT}(k)$. The latter is clear since $\op{MT}(k)$ obviously  is equivalent
to the category of finite-dimensional graded vector  spaces. 

Condition (4) is established in \prettyref{lem:cond4}. As mentioned
above, the short exact sequences before \cite[Corollaire 1.4.24]{BBD}
imply that the kernel of  $j_!\mathbb{Q}[d]\to
j_{!\ast}\mathbb{Q}[d]$ and the cokernel of $j_{!\ast}\mathbb{Q}[d]$
are supported on the complement of $U$, whence Condition (5). 

Finally, \cite[Lemma 3.2.4]{BBD} allows to reduce Condition (6) to the
vanishing of \prettyref{lem:vanish} below. 
\end{proof}

\begin{lemma}
\label{lem:vanish}
Assume the situation in \ref{Real-ladic} or \ref{Real-hodge}, and  let
$(X,\mathcal{S})$ be 
an affinely Whitney--Tate stratified variety. For any two strata
$j_t:X_t\to X$ and $j_s:X_s\to X$ and $(n,a)\neq (0,0)$, we have 
$$
\DMT_{\mathcal{S}}(j_{t!}\underline{X_t},
j_{s*}\underline{X_s}(a)[n])=0.
$$
\end{lemma}

\begin{proof}
If $X_t\neq X_s$, then $j_t^!j_{s\ast}=0$ implies the vanishing
directly. If $X_t=X_s$, then remark first $j_t^!j_{t\ast}=\op{id}$,
since this holds as well for an open as for a closed embedding.
Thus we are reduced to showing 
$
\DMT(\underline{D},
\underline{D}(a)[n])=0
$
for $(n,a)\neq (0,0)$ and $D$ an affine space, and this follows from
homotopy invariance and \ref{eqf}. 
\end{proof}



\begin{theorem}
\label{thm:cato}
    Let $(X,\mathcal S)$ be an affinely Whitney--Tate stratified
    variety.  
\begin{enumerate}
\item In the situation of \ref{Real-ladic},
    the realization functor
$$\op{Real}_\ell: \PMT_{\mathcal S}(X;\DQ_\ell)\ra
\op{Perv}_{\mathcal S}(X\times_k\bar k;\DQ_\ell)$$ considered in \ref{raed} is a degrading functor in the sense of
\cite{BGSo}.
\item In the situation of \ref{Real-hodge},
    the realization functor
$$\op{Real}_H: \PMT_{\mathcal S}(X;\mathbb{C})\ra
\op{Perv}_{\mathcal S}(X;\mathbb{C})$$
 considered in \ref{raed} is a
degrading functor in the sense of 
\cite{BGSo}.
\end{enumerate}
\end{theorem}

\begin{proof}
Again, we only consider the $\ell$-adic realization, the Hodge
realization argument being similar.
We need to show that the induced functor
$$\op{Real}_\ell: \op{Der}^{\op{b}}(\PMT_{\mathcal S}(X;\DQ_\ell))\ra \op{Der}^{\op{b}}(\op{Perv}_{\mathcal S}(X\times_k\bar
k;\DQ_\ell))$$ induces isomorphisms
$$ \bigoplus_{n\in\DZ}\op{Der}(P, Q(n))
\sira \op{Der}(\op{Real}P, \op{Real}Q)$$
for any complexes $P,Q$. But since there are enough projectives, by
\prettyref{prop:PrOn}, and these clearly go to projectives, we just
need to show the analogous statement for the functor
$\op{Real}: \op{Hot}^{\op{b}}(p\PMT_{\mathcal S}(X;\DQ_\ell))\ra \op{Hot}^{\op{b}}(p\op{Perv}_{\mathcal S}(X\times_k\bar
k;\DQ_\ell))$ on the bounded homotopy category of projective objects.
For single projective objects however we already know it from  
 \prettyref{thm:fulg}, and from there the extension to the
bounded homotopy categories is immediate.
\end{proof}

\begin{theorem}
\label{thm:PrOoi}
Assume the situation in \ref{Real-ladic} or \ref{Real-hodge}.
Let $(X,\mathcal S)$ be an 
affinely Whitney--Tate stratified variety
and let $p\PMT_{\mathcal S} (X)$ be the additive category of projective 
perverse objects. Then
we get equivalences of categories 
 $$\op{Der}^{\op{b}}(\PMT_{\mathcal S} (X))
\stackrel{\approx}{\leftarrow}\op{Hot}^{\op{b}}(p\PMT_{\mathcal S} (X))
\sirra \DMT_{\mathcal S} (X)$$ 
by the obvious functor towards the left and a tilting functor
as in \prettyref{prop:rvt} towards the right.
\end{theorem}


\begin{proof}
  The equivalence to the left follows easily from \ref{prop:PrOn}.
To obtain the tilting equivalence, it will be sufficient to show 
$\DMT_{\mathcal S}(P,M[n])=0$ for
$P, M\in \PMT_{\mathcal S} (X)$ with $P$ projective and $n\neq 0$.
By an induction on a $\Delta$-flag of our projective $P$
we deduce
 $\DMT_{\mathcal S}(P,\nabla_t[n])=0$ for $n\neq 0$.
By an induction on a $\nabla$-flag, we  get
$\DMT_{\mathcal S}(P,I[n])=0$ for $n\neq 0$
and any injective object $I\in \PMT_{\mathcal S} (X)$.
Now remember we needed  to show 
$\DMT_{\mathcal S}(P,M[n])=0$ for
$P, M\in \PMT_{\mathcal S} (X)$ with $P$ projective and $n\neq 0$.
For $n<0$ or $n=1$ this is clear anyhow. 
Thus if it is ok for two terms of a short exact sequence, it is also
ok for the third term. Thus if it is ok for all terms of
a finite resolution of a given object, it will also be ok for
the given object itself. But  by  \prettyref{lem:vanish} 
it is ok  for injective objects, and every object has
a finite injective resolution.  
\end{proof}

\begin{corollary}\label{bhu}
Assume the situation in \ref{Real-ladic} or \ref{Real-hodge}.
  Let $(X,\mathcal S)$ be an 
affinely Whitney--Tate stratified variety.
Then:
\begin{enumerate}
\item All simple perverse motives $\mathcal L\in\PMT_{\mathcal S}(X)$ are
  up to a shift in the heart of the weight structure, in formulas $\mathcal
  L\in(\DMT_{\mathcal S}(X))_{w=p}$ for some $p\in \DZ$;
\item
All perverse motives $\mathcal L\in\PMT_{\mathcal S}(X)$, 
which are pure of a given weight, are semisimple;
\item
All pure motives $\mathcal L\in\DMT_{\mathcal S}(X)_{w=p}$
are isomorphic to the direct sum of their perverse cohomology objects, which
in
turn are perverse semisimple.
\end{enumerate}

\end{corollary}
\begin{proof}
  The first two points follow 
using \prettyref{thm:cato}
from the analogous result for $\ell$-adic sheaves 
in \cite{BBD} by applying a suitable realization functor. 
The same argument shows that all perverse cohomology objects of a pure
object are pure, more precisely given 
$\mathcal F\in\DMT_{\mathcal S}(X)_{w=0}$ 
we have 
$^p\mathcal H^n\mathcal F\in \DMT_{\mathcal S}(X)_{w=-n}$
for the $n$-th perverse cohomology object. This in turn says 
by the definition of a weight structure, that the triangles inductively 
putting together the object $\mathcal F$ from its perverse cohomology objects
all have the relevant map zero and so the object $\mathcal F$ has to be
the direct sum of its perverse cohomology objects.
\end{proof}
\begin{Bemerkungl}\label{Koszul} 
 If $(X,\mathcal S)$ is an affinely Whitney--Tate stratified variety, whose
  pure objects are even pointwise pure, we deduce that 
 the category of perverse motives $\PMT_{\mathcal S}(X)$
 has the  ``Koszul property'': Given two simple objects $\mathcal N,\mathcal
 M$
of weights $ n,m$, the only nonzero extensions from the first to the second  
with respect to the abelian category  
$\PMT_{\mathcal S}(X)$ are in $\op{Ext}^{n-m}(\mathcal N,\mathcal
 M)$. To see this, one may use \prettyref{thm:PrOoi}
to identify the Ext-group in question with $\mathscr{T}(\mathcal N,\mathcal
 M[n-m])$ and then compute it by the spectral sequence explained in 
\cite[3.4.1]{BGSo}. In particular $\PMT_{\mathcal S}(X)$ is then, up to
formally adding a square root of the Tate twist, equivalent to the category of
finite dimensional modules over a Koszul ring of finite dimension over $\DQ$.
All this is, up to the very satisfying interpretation by true motives,
already contained in \cite{BGSo}. The details are left to the reader.
\end{Bemerkungl}

\appendix

\section{
Stratifications and singularities with mixed
  Tate resolutions}
\label{sec:WTS}

In the following section, we provide another criterion for a
stratified variety $(X,\mathcal{S})$ to be Whitney--Tate. This will
also provide another proof of \prettyref{prop:erpt2}. In the following
section, we fix a motivic triangulated category $\mathscr{T}$. 

We first recall \cite[Theorem 4.4]{wildeshaus:intermediate}:

\begin{proposition}
Let $(X,\mathcal{S})$ be a smooth stratified scheme such that for each
stratum $X_s$, the closure $\overline{X_s}$ is also smooth. Then for each
pair of strata $X_s$ and $X_t$ with $i_t:X_t\hookrightarrow
\overline{X_s}$, the compositions 
$
i_t^\ast\circ j_\ast, i_t^!\circ j_!:\mathscr{T}(X_s)\rightarrow
\mathscr{T}(X_t) 
$
preserve the triangulated subcategories of mixed Tate motives. 
In particular, the stratification is Whitney--Tate.  
\end{proposition}

Next, we provide a criterion for being Whitney--Tate which covers some
cases where  closures of strata are not regular. 
In this more general setting, we can still argue as in the smooth case
after a suitable resolution of singularities. This, however, requires
that there is a resolution such that the motives of the
fibers of the resolution are mixed Tate.  The condition is a motivic
version of condition $(\ast)$ in \cite[Section 1.4]{BGSo}.

\begin{proposition}
Let $(X,\mathcal{S})$ be a stratified scheme. Assume
that for each stratum $X_s$ of $X$ there exists a resolution of
singularities 
$\rho_s:\widetilde{X_s}\rightarrow\overline{X_s}$ with the following
properties: 
\begin{enumerate}
\item $\rho_s$ is surjective and proper, 
\item $\widetilde{X_s}$ is smooth, 
\item for each stratum $X_t\hookrightarrow\overline{X_s}$, the restriction
$\widetilde{X_s}\times_{\overline{X_s}}X_t$  satisfies 
$$
\op{M}_{X_t}\left(\widetilde{X_s}\times_{\overline{X_s}}X_t\right)\in
    \DMT(X_t). 
$$
\end{enumerate}

Then for each pair of
strata $X_s$ and $X_t$ with $i_t:X_t\hookrightarrow \overline{X_s}$,
the compositions
$$
i_t^\ast\circ j_\ast, i_t^!\circ j_!:\mathscr{T}(X_s)\rightarrow \mathscr{T}(X_t) 
$$
preserve the triangulated subcategories of mixed Tate motives.  
In particular, the stratification is Whitney--Tate.  
\end{proposition}

\begin{proof}
Since $i_t^\ast j_\ast$ is dual to $i_t^!j_!$ and the motivic duality
restricts to $\DMT(X_t)$, it suffices to prove one of the
assertions.  Since $i_t^\ast j_\ast$ is compatible with Tate twists,
it suffices to prove that $i_t^\ast j_\ast\mathbb{Q}\in
\DMT(X_t)$. 

For the proof, we now fix $X_t$. Without loss of generality, we can
assume that $X_t$ is closed in $X$. If not, we consider the scheme
$X\setminus (\overline{X_t}\setminus X_t)$. This still satisfies the
conditions, and the closed complement $\overline{X_t}\setminus X_t$
does not enter the computations of the compositions of functors.

The proof that $i_t^\ast \circ(j_s)_\ast\mathbb{Q}$ is mixed Tate now
proceeds by induction on the dimension of $X_s$. Therefore, assume
that for all $X_r$ with $X_r\hookrightarrow\overline{X_s}$, the claim is
satisfied, i.e. $i_t^\ast\circ (j_r)_\ast\mathbb{Q}$ is mixed Tate. 
Now consider the following diagram:
\begin{center}
  \begin{minipage}[c]{10cm}
    \xymatrix{
      \widetilde{X_s}\times_{\overline{X_s}} X_t \ar[r]^{i_t'} \ar[d]_{p_t} &
      \widetilde{X_s} \ar[d]^p & 
      \widetilde{X_s}\times_{\overline{X_s}}X_s \ar[l]_{j'} \ar[d]^{p_s}\\ 
      X_t\ar[r]_{i_t} & \overline{X_s} &X_s \ar[l]^j
    }
  \end{minipage}
\end{center}
The arrow $p$ in the middle is the resolution of singularities
provided by the assumption. The rest of the diagram consists of
restricting $p$ to the strata $X_s$ and $X_t$. 

Proper base change for the left square states $i_t^\ast p_\ast M\cong
(p_t)_\ast (i_t')^\ast M$. Then it suffices to show that $(i_t')^\ast
p^\ast j_\ast$ is in
$\DMT(\widetilde{X_s}\times_{\overline{X_s}} X_t)$, and
that $(p_t)_\ast$ preserves mixed Tate motives.

The fact that $(p_t)_\ast$ preserves mixed Tate motives follows from
part (3) of our assumption: as $(p_t)_\ast$ commutes with Tate twists,
it suffices to show that $(p_t)_\ast\mathbb{Q}$ is contained in 
$\DMT(X_t)$. But by assumption,
$$
(p_t)_\ast\mathbb{Q}\cong
\op{M}_{X_t}(\widetilde{X_s}\times_{\overline{X_s}} X_t) \in
\DMT(X_t).
$$

To prove that $(i_t')^\ast p^\ast j_\ast$ is in
$\DMT(\widetilde{X_s}\times_{\overline{X_s}} X_t)$, we
employ the localization sequence in the situation $X=\widetilde{X_s}$,
$i_t':Z=\widetilde{X_s}\times_{\overline{X_s}}X_t\hookrightarrow X$ and
$j':U=X\setminus Z\hookrightarrow X$.  
In that situation, the localization sequence for $\mathbb{Q}\in
\mathscr{T}(\widetilde{X_s})$ has the form
$$
(i_t')_\ast (i_t')^!\mathbb{Q}\rightarrow\mathbb{Q}\rightarrow (j')_\ast
(j')^\ast\mathbb{Q}\rightarrow (i_t')_\ast (i_t')^!\mathbb{Q}[1].
$$
But $(j')^\ast\mathbb{Q}_X\cong \mathbb{Q}_U$ and
$(i_t')^\ast\mathbb{Q}_X\cong \mathbb{Q}_Z$. By absolute purity, we have
$(i_t')^!\mathbb{Q}_Z\cong\mathbb{Q}(-d)[-2d]$ with $d$ the
codimension of $Z$ in $X$. Restricting this sequence using $(i_t')^\ast$
provides the following triangle in $\mathscr{T}(Z)$:
$$
\mathbb{Q}_Z(-d)[-2d]\rightarrow \mathbb{Q}_Z\rightarrow (i_t')^\ast (j')_\ast
\mathbb{Q}_U\rightarrow \mathbb{Q}_Z(-d)[-2d+1].
$$

It suffices to show that the difference between the motives $(p_t)_\ast
(i_t')^\ast (j')_\ast\mathbb{Q}_U$ and $(p_t)_\ast (i_t')^\ast p^\ast(j_s)_\ast
\mathbb{Q}_{X_s}$ is mixed 
Tate. The preimages $p^{-1}(X_r)$ provide a stratification of
$\widetilde{X_s}$ by part (2) of the assumption. Inductively applying
a  localization argument similar to the one used in \cite[Theorem
4.4]{wildeshaus:intermediate}  to $U=\widetilde{X_s}\setminus Z$ by
taking out smooth closed strata, we see that the difference between
$j_\ast \mathbb{Q}_U$ and $p^\ast(j_s)_\ast 
\mathbb{Q}_{X_s}$ is given by extensions of mixed Tate motives on the
strata $p^{-1}(X_r)$. Therefore, it suffices to show that for each
stratum $X_r$ in $\overline{X_s}$, the functor $(p_t)_\ast\circ
(i'_t)^\ast \circ
(j_r')_\ast:\mathscr{T}(\widetilde{X_s}\times_{\overline{X_s}}X_r)\rightarrow  
\mathscr{T}(X_t)$ preserves mixed Tate motives. By the inductive
assumption, this is true for 
$i_t^\ast\circ
(j_r)_\ast:\mathscr{T}(X_r)\rightarrow\mathscr{T}(X_t)$, and by part
(3) of our assumption, it is also true for $i_t^\ast\circ
(j_r)_\ast\circ (p_r)_\ast$. Obviously $(j_r)_\ast\circ
(p_r)_\ast\cong p_\ast\circ (j'_r)_\ast$. By proper base change,
$i_t^\ast\circ p_\ast\cong (p_t)_\ast\circ (i_t')^\ast$. Combining
these, we find that $(p_t)_\ast\circ (i_t')^\ast\circ (j_r')_\ast\cong
i_t^\ast\circ (j_r)_\ast\circ (p_r)_\ast$. Hence, this latter
composition preserves mixed Tate motives, which finishes the proof.
\end{proof}

\begin{proposition}
\label{prop:erpt1}
Let $G$ be a split reductive group, let $B\subset G$ be a Borel
subgroup, and denote by $(B)$ the stratification of $G/B$ by Schubert
cells. Then for each $w\in W$, the Bott--Samelson 
resolution $\rho_w:BS(w)\rightarrow\overline{X_w}$ of
Demazure--Hansen has the following 
properties: 
\begin{enumerate}
\item $\rho_w$ is surjective and proper, 
\item $BS(w)$ is smooth, 
\item for each $v\in W$ with $X_v\in\overline{X_w}$, the restriction
$BS(w)\times_{\overline{X_w}}X_v$ satisfies 
$$
\op{M}_{X_v}\left(BS(w)\times_{\overline{X_w}}X_v\right)\in
    \DMT(X_v). 
$$
\end{enumerate}
In particular, the Bruhat stratification of a 
flag variety is Whitney--Tate. 
\end{proposition}

\begin{proof}
Properties (1) and (2) are well known. Property (3) follows by
iterative use of the localization sequence once we can show that for
each point $x$ of $\overline{X_w}$, the fibre of the Bott--Samelson
resolution $\rho_w^{-1}(x)$ has a paving by affine spaces. This is the
case, as discussed in \cite{haines}. 
\end{proof}

\section{General tilting equivalences}
\label{sec:tilting}

Here we formulate a general tilting-type theorem. 
Possible sources for statements of this type are \cite{RickMo,
  Kel}. We will sketch a proof and discuss the  tilting functor. 
\begin{proposition}
\label{prop:rvt} 
Let  $\mathcal A$ be an abelian
category and  $(T_i)_{i\in I}$ a family of complexes in 
$ \op{Hot}(\cal A)$ 
 such that for all $i,j\in I$ and $ n\in\DZ$ we have  
$\op{Hot}_{\cal A}(T_i,T_j[n])\sira \op{Der}_{\cal A}(T_i,T_j[n])$ and
$$\op{Der}_{\mathcal A}(T_i,T_j[n])\neq 0\;\RA \; n=0
$$
Then the  embedding of the full additive subcategory
of $
\op{Der}(\mathcal A)$ 
 generated by the  objects 
$T_i$
 can be extended to 
a fully faithful 
triangulated functor 
  $$\op{Hot}^{\op{b}}\left(\op{add}( T_i\mid i\in I)\right)
\stackrel{\sim}{\hra} 
\op{Der}(\mathcal A)$$ 
\end{proposition}

\begin{proof}
For simplicity let us first consider the case of a finite family of 
objects $T_1,\ldots, T_r$. 
Let us consider the complex $T = \bigoplus_i T_i$. Its
endomorphism complex 
$$E := \op{End}(T)\pdef\bigoplus_{n}\mathcal{A}(T,T[n])$$
 has a natural structure of a 
dg-ring with idempotents $1_i \in E$ given by the
projection to each factor. Then the 
localization functor induces by devissage an equivalence 
between the full
triangulated subcategories
\[
\langle T_1, \dots, T_r \rangle_{\Delta} 
\]
generated by the objects $T_i$ 
in $ \op{Hot}(\mathcal{A})$ and $ \op{Der}(\mathcal{A})$ respectively.
On the other hand the
functor $\op{Hom}(T, \;)$
induces an equivalence from the first of these
 triangulated categories  to the full triangulated
subcategory
\[
\langle 1_1 E, \dots, 1_r E \rangle_{\Delta} \subset \op{dgDer-} E
\]
generated by the right dg-modules $1_i
E$ in the localization $\op{dgDer-} E$
of the category of right dg-modules over $E$ by quasi-isomorphisms. 
Now recall that for any quasi-isomorphism $D
\stackrel{\sim}{\longrightarrow} E$ of 
dg-rings  the restriction induces an equivalence of
triangulated categories
\[
\op{dgDer-} E  \sirra  \op{dgDer-}D
\]
Up to this point we did not need the  condition
$\op{Der}_{\mathcal A}(T_i,T_j[n])\neq 0\RA  n=0
$.
This additional assumption however implies 
that the cohomology $\mathcal H E$ of $E$ is concentrated in
degree zero. We therefore have  quasiisomorphisms
$$
\mathcal H E\quad\stackrel{\sim}{\longleftarrow} \quad \mathcal
Z^{0} E \oplus E^{<0}\quad \stackrel{\sim}{\longrightarrow} \quad E
$$
of dg-rings. Let us abbreviate  $H\pdef \mathcal H E$.  
Under the equivalence of triangulated categories 
$$
\op{Der}(\op{mod-}H)= \op{dgDer-} H
\sirra  \op{dgDer-}E 
$$
defined by our quasi-isomorphisms the objects $1_iH$ will correspond to $1_iE$.
In addition, the localization functor 
$\op{Hot}(\op{mod-}H)  \ra \op{Der}(\op{mod-}H)$ 
induces, again by devissage, an equivalence between the full
triangulated subcategories  generated  by the right modules $1_i H$ in
both of these triangulated categories.  
The first of these triangulated categories
in turn coincides with the homotopy category
$$\op{Hot}^{\op{b}}(\op{add}(1_1 H,\dots,1_rH))$$
 of the full additive subcategory
$\op{add}(1_1 H,\dots,1_rH)\subset \op{mod-}H$ generated by
our right $H$-modules. 
Now sure enough the obvious maps give  isomorphisms
$1_iH1_j\sira\op{Mod}_{H}(1_jH, 1_iH)$ to the space of 
homomorphisms of right $H$-modules  
and $1_iH1_j\sira\op{Der}_{\mathcal A}(T_j, T_i)$. 
This gives us an equivalence 
$$\op{add}(1_1 H,\dots,1_rH)\sirra \op{add}(T_1 ,\dots,T_r)$$
 of additive categories and finishes the proof of the proposition
in the case of a finite family of objects. 

The general case follows similar lines. Instead of a single generator,
we have to consider categories enriched in abelian
groups. Objects like these are called ringoids or rings with many
objects in the literature. The usual definitions of modules still
apply to rings with many objects, and the above proof works in that
setting. More details can be found in  \cite{Kel}. 
\end{proof}

\begin{Bemerkungl}
\label{rem:tilt}
  Given objects $\bar T_i\in\op{Der}(\mathcal A)$ with 
$(\op{Der}_{\mathcal A}(\bar T_i,\bar T_j[n])\neq 0\;\RA \; n=0
)$ we can quite often find  representatives $T_i\in\op{Hot}(\mathcal A)$
with the properties required in the Proposition
by choosing some kind of projective or injective resolutions. 
\end{Bemerkungl}

\begin{Bemerkungl}
\label{rem:tiltue}
In the setting of \prettyref{prop:rvt}, suppose in addition that $\op{Hot}\mathcal A$ and $\op{Der}\mathcal A$ admit countable direct sums, that the localization functor preserves those, and that all the objects $T_i$ are compact in $\op{Hot}\mathcal A$ and $\op{Der}\mathcal A$. Then the embedding $\op{add}^\infty( T_i\mid i\in I)\subset \op{Der}(\mathcal A)$ of the full additive subcategory consisting of all countable direct sums of copies of objects among the $T_i$  can be extended to a fully faithful triangulated functor 
$$
\op{Hot}^{\op{b}}\left(\op{add}^\infty( T_i\mid i\in I)\right) \stackrel{\sim}{\hra} \op{Der}(\mathcal A)
$$ 
The argument stays essentially the same. The conditions that $\op{Hot}\mathcal A$ and $\op{Der}\mathcal A$ admit countable direct sums and that the localization functor preserves those are satisfied  for example in the case where $\mathcal A$ is a category of sheaves, since in this case a right adjoint for the localization functor can be obtained by choosing  K-injective resolutions following \cite{Spa}.    
\end{Bemerkungl}

\bibliographystyle{amsalpha}\bibliography{pub}

\end{document}